\documentclass[reqno]{amsart}
\usepackage{amssymb}
\usepackage{graphicx}

\usepackage[usenames, dvipsnames]{color}
\usepackage{verbatim}
\usepackage{mathrsfs}
\usepackage{bm}
\usepackage{cite}

\numberwithin{equation}{section}

\newtheorem{theorem}{Theorem}[section]

\newtheorem{lemma}[theorem]{Lemma}
\newtheorem{prop}[theorem]{Proposition}

\theoremstyle{definition}
\newtheorem{remark}[theorem]{Remark}

\theoremstyle{definition}

\theoremstyle{definition}

\makeatletter
\def\dashint{\operatorname%
{\,\,\text{\bf-}\kern-.98em\DOTSI\intop\ilimits@\!\!}}
\makeatother

\def\\det{\text{\det}}

\def\.5{\frac{1}{2}}

\newcommand{\RN}[1]{%
  \textup{\uppercase\expandafter{\romannumeral#1}}%
}

\renewcommand{\epsilon}{\varepsilon}

\newcounter{marnote}

\begin{document}

\title[Regularity and stability for solutions]{Regularity and stability for solutions to elliptic equations and systems arising from high-contrast composites}

%\author[Z. Yang]{Zhuolun Yang}
%
%\address[Z. Yang]{Institute for Computational and Experimental Research in Mathematics, Brown University, 121 South Main Street, Providence, RI 02903, USA.}
%%\address{2. School of Mathematical Sciences, Beijing Normal University, Beijing 100875, China.}
%\email{zhuolun\_yang@brown.edu}

\author[Z. Zhao]{Zhiwen Zhao}

\address[Z. Zhao]{Beijing Computational Science Research Center, Beijing 100193, China.}
%\address{2. School of Mathematical Sciences, Beijing Normal University, Beijing 100875, China.}
\email{zwzhao365@163.com}

%\footnote{}

\date{\today} % delete this line to display the current date

%%% BEGIN DOCUMENT

\maketitle

\begin{abstract}
The main objective of this paper is to study the regularity and stability for solutions to the conductivity problems with degenerate coefficients in the presence of two rigid conductors, as one conductor keeps motionless and another conductor moves in some direction by a sufficiently small translational distance. We will show that the solutions are smooth and stable with respect to the small translational distance. Our results contain the following three cases: two perfect conductors, two insulators, a perfect conductor and an insulator. Further, we extend the results to the elasticity problem modeled by the Lam\'{e} system with partially infinite coefficients.

\end{abstract}

\maketitle
%\date{}
%\maketitle
%{\bf Abstract}

%{Keywords}

%{\bf Mathematics Subject Classification(2010)} $35{\rm B}33 \cdot   35{\rm B}40 \cdot  35{\rm B}44$

\section{Introduction and main results}

This investigation is stimulated by the problem of damage and fracture in high-contrast fiber-reinforced composites. The material fracture occurs in the thin gaps between fibers, since there always appear high concentrated fields including extreme electric and stress fields in these zones. These physical fields can be described by the gradients of solutions to elliptic equations and systems with discontinuous coefficients. Much effort has been devoted to quantitative analysis for the behavior of the gradients in the narrow regions between rigid inclusions since Babu\u{s}ka et al's famous numerical work \cite{BASL1999}. In the case of nondegenerate finite coefficients, it has been proved in \cite{DL2019,LV2000,LN2003,BV2000} that the gradients of solutions remain bounded independent of the distance $\varepsilon$ between inclusions. Especially in \cite{LN2003}, Li and Nirenberg established stronger $C^{1,\alpha}$ estimates for general second-order elliptic systems with piecewise H\"{o}lder continuous coefficients, which completely demonstrates the numerical observation in \cite{BASL1999}.
%To be specific, when the conductivity $k$ tends to $0$ or $\infty$, the conductivity problem with piecewise constant coefficients becomes the insulated or perfect conductivity problems and in these two cases the gradients of solutions can blow up with respect to the distance $\varepsilon$.

However, the situation becomes significantly different when the coefficients degenerate to be zero or infinity. Let $D\subseteq\mathbb{R}^{n}\,(n\geq3)$ be a bounded domain with $C^{2,\alpha}$ boundary, whose interior contains two $C^{2,\alpha}\,(0<\alpha<1)$ inclusions $D_{1}$ and $D_{2}$ with $\varepsilon$ apart, where $\varepsilon$ is a sufficiently small positive constant. Suppose further that $D_{i}$, $i=1,2$ are far away from the external boundary $\partial D$. When the conductivity $k$ degenerates to be $0$ or $\infty$, we obtain the following three types of boundary value problems (see \cite{BLY2009,BLY2010}):
\begin{align}\label{con006}
\begin{cases}
\Delta u=0,&\mathrm{in}\;D\setminus\overline{D_{1}\cup D_{2}},\\
u|_{+}=u|_{-},&\mathrm{on}\;\partial{D}_{i},\,i=1,2,\\
u=C_{i}^{0},&\mathrm{on}\;\overline{D}_{i},\,i=1,2,\\
\int_{\partial D_{i}}\frac{\partial u}{\partial\nu}\big|_{+}=0,&i=1,2,\\
u=\varphi,&\mathrm{on}\;\partial D,
\end{cases}
\end{align}
and
\begin{align}\label{PRO002}
\begin{cases}
\Delta u=0,&\mathrm{in}\;D\setminus(\partial D_{1}\cup\partial D_{2}),\\
u|_{+}=u|_{-},&\mathrm{on}\;\partial{D}_{i},\,i=1,2,\\
\frac{\partial u}{\partial\nu}\big|_{+}=0,&\mathrm{on}\;\partial{D}_{i},\,i=1,2,\\
u=\varphi,&\mathrm{on}\;\partial D,
\end{cases}
\end{align}
and
\begin{align}\label{PRO09}
\begin{cases}
\Delta u=0,&\mathrm{in}\;D\setminus(\partial D_{1}\cup\overline{D}_{2}),\\
u|_{+}=u|_{-},&\mathrm{on}\;\partial{D}_{i},\,i=1,2,\\
\frac{\partial u}{\partial\nu}\big|_{+}=0,&\mathrm{on}\;\partial{D}_{1},\\
u=C^{0},&\mathrm{on}\;\overline{D}_{2},\\
\int_{\partial D_{2}}\frac{\partial u}{\partial\nu}\big|_{+}=0,\\
u=\varphi,&\mathrm{on}\;\partial D,
\end{cases}
\end{align}
where $u\in H^{1}(D)$, $\varphi\in C^{2,\alpha}(\overline{D})$, $\varphi\not\equiv0$ on $\partial D$, $C_{i}^{0}$, $i=1,2,$ are the free constants determined by the fourth line of \eqref{con006}, the free constant $C^{0}$ is determined by the fourth line of \eqref{PRO09}, $\nu$ denotes the unit outer normal to the domain. Here and throughout this paper the subscript $\pm$ shows the limit from outside and inside the domain, respectively. It is worth pointing out that problems \eqref{con006} and \eqref{PRO002} are, respectively, called the perfect and insulated conductivity problems and their gradients of solutions always blow up with respect to the distance $\varepsilon$ between inclusions, as the distance $\varepsilon$ tends to zero. By the elliptic regularity theory, we know that $u\in C^{2,\alpha}(\overline{D\setminus D_{1}\cup D_{2}})$ for the perfect conductivity problem \eqref{con006}, $u\in C^{2,\alpha}(\overline{D\setminus D_{1}\cup D_{2}})\cap C^{2,\alpha}(\overline{D_{1}\cup D_{2}})$ for the insulated conductivity problem \eqref{PRO002}, and $u\in C^{2,\alpha}(\overline{D\setminus D_{1}\cup D_{2}})\cap C^{2,\alpha}(\overline{D}_{1})$ for the conductivity problem \eqref{PRO09} with a perfect conductor and an insulator.

With regard to the perfect conductivity case, there is a long list of literature involving different techniques and cases based on shapes of inclusions, dimensions, and applied boundary conditions. In the presence of two strictly convex inclusions, the gradient blow-up rates for the perfect conductivity problem have been proved to be $\varepsilon^{-1/2}$ \cite{AKLLL2007,BC1984,BLY2009,AKL2005,Y2007,Y2009,K1993} for $n=2$, $|\varepsilon\ln\varepsilon|^{-1}$ \cite{BLY2009,LY2009,BLY2010,L2012} for $n=3$ and $\varepsilon^{-1}$ \cite{BLY2009} for $n\geq4$, respectively. For more precise characterizations on the singular behavior of the electric field concentration, see \cite{ACKLY2013,LLY2019,BT2013,KLY2013,KLY2014} and the references therein. In the case when the current-electric field relation is nonlinear, we refer to \cite{G2012,CS2019,CS201902}. In addition, similar results have also been extended to the Lam\'{e} system with partially infinite coefficients, see \cite{BLL2015,BLL2017,KY2019} and the references therein. Different from the perfect conductivity case, it has been recently shown in \cite{DLY2022} that the gradient blow-up rate for the insulated conductivity problem also depends on the principal curvature of the surfaces of inclusions besides the dimension. For more related studies on the insulated conductivity problem, see \cite{AKL2005,AKLLL2007,BLY2010,Y2016,W2021,LY2021,DLY2021,LY202102}. As for problem \eqref{PRO09}, Dong and Yang \cite{DY2022} recently proved that there appears no blow-up for the gradient of the solution, which partially answers a question raised by Kang \cite{K2022}.

The above-mentioned work mainly focus on the establishment of gradient estimates and asymptotics in the thin gaps between inclusions and aim to reveal the dependence on the distance between inclusions. By contrast, the distance between two inclusions considered in this paper is of constant order but not a infinitely small quantity. In this case, we dedicate to studying the regularity and stability of solutions, as one inclusion moves in some direction for a sufficiently small distance and another inclusion keeps motionless. We will prove that the solutions are smooth and stable with respect to the small translational distance.

\subsection{Main results}
Although the results of this paper can be extended to any finite number of inclusions of general convex shapes, we restrict to two regular curvilinear squares and cubes for the convenience of presentation. For $x^{0}\in\mathbb{R}^{n}$, $r>0$ and $m,n\geq2$, let $\mathcal{B}_{r}^{(m)}(x^{0})$ represent the curvilinear squares and cubes centered at $x^{0}$ with the radius $r$ as follows:
\begin{align*}
\sum^{n}_{i=1}|x_{i}-x_{i}^{0}|^{m}=r^{m}.
\end{align*}
In particular, when $m=2$, it is a circle or ball, that is, $\mathcal{B}_{r}^{(2)}(x^{0})=B_{r}(x^{0})$. In the following, let
\begin{align}\label{DOM001}
D=\mathcal{B}_{10}^{(m)}(0),\;\, D_{1}^{h}=\mathcal{B}^{(m)}_{1}((3+h)e_{n}),\;\, D_{2}=\mathcal{B}^{(m)}_{1}(0),\;\,\Omega_{h}=D\setminus\overline{D_{1}^{h}\cup D_{2}},
\end{align}
where $|h|\ll1/2$ and $e_{n}=(0',1)$. Here and throughout this paper, we use superscript prime to denote $(n-1)$-dimensional variables. For $m\geq2$, define
\begin{align}\label{alpha}
\gamma=
\begin{cases}
\min\{\alpha,m-2\},&2<m<3,\\
\alpha,&m=2\;\mathrm{or}\;m\geq3.
\end{cases}
\end{align}

For any two subdomains $A,B\subseteq \mathbb{R}^{n}$, denote $A\Delta B:=(A\setminus \overline{B})\cup(B\setminus\overline{A})$. The regularity and stability results for these above three types of conductivity problems are stated as follows.
\begin{theorem}\label{thm001}
For $n\geq2$ and $|h|\ll1/2$, let $u_{h}\in H^{1}(D)\cap C^{2,\gamma}(\overline{\Omega}_{h})$ be, respectively, the solutions of problems \eqref{con006}, \eqref{PRO002} and \eqref{PRO09} with $D,D_{1}^{h},D_{2},\Omega_{h}$ defined by \eqref{DOM001} and $\gamma$ given by \eqref{alpha}. Then there exists a small constant $h_{0}>0$ such that $u_{h}$ are smooth and stable in $h$ for $h\in(-h_{0},h_{0})$. Moreover, $\|u_{h}-u_{0}\|_{L^{\infty}(\overline{D})}=O(|D_{1}^{h}\Delta D_{1}^{0}|)$.
\end{theorem}
\begin{remark}
Observe that $\partial D_{1}^{h}$, $\partial D_{2}$ and $\partial D$ are $C^{\infty}$ if $m\geq2$ is an even number. Then in the case when $\varphi\in C^{\infty}(\overline{D})$ and $m\geq2$ is an even number, by applying the proof of Theorem \ref{thm001} with a slight modification, we can obtain that for any $k\geq0$, $\|u_{h}-u_{0}\|_{C^{k}(\overline{D\setminus D_{1}^{0}\cup D_{1}^{h}\cup D_{2}})}=O(|h|)\rightarrow0$, as $h\rightarrow0$. In addition, when these two inclusions move in some directions simultaneously by a small distance, the similar results can be also obtained by slightly modifying the proofs. For example, if $D_{2}$ is replaced by $D_{2}^{h}:=\mathcal{B}^{(m)}_{1}(\pm h e_{n})$, we can obtain that $\|u_{h}-u_{0}\|_{L^{\infty}(\overline{D})}=O(|(D_{1}^{h}\Delta D_{1}^{0})\cup(D_{2}^{h}\Delta D_{2}^{0})|)$.

\end{remark}

The results can be also extended to the elasticity problem modeled by the Lam\'{e} system with partially infinite coefficients. Specifically, consider the following boundary value problem:
\begin{align}\label{La.002}
\begin{cases}
\partial_\alpha(A_{ij}^{\alpha\beta}\partial_\beta u_{h}^j)=0,\quad&\mathrm{in}\;D\setminus\overline{D_{1}^{h}\cup D_{2}},\\
u_{h}|_{+}=u_{h}|_{-},&\mathrm{on}\ \partial{D}_{1}^{h}\cup\partial D_{2},\\
u_{h}=\sum^{\frac{n(n+1)}{2}}_{\alpha=1}C_{1\alpha}^{h}\phi_{\alpha},&\mathrm{on}~\overline{D_{1}^{h}},\\
u_{h}=\sum^{\frac{n(n+1)}{2}}_{\alpha=1}C_{2\alpha}^{h}\phi_{\alpha},&\mathrm{on}~\overline{D}_{2},\\
\int_{\partial{D}_{1}^{h}}A_{ij}^{\alpha\beta}\partial_\beta u_{h}^j\nu_{\alpha}\phi_{k}^{i}|_{+}=0,&k=1,2,...,\frac{n(n+1)}{2},\\
\int_{\partial{D}_{2}}A_{ij}^{\alpha\beta}\partial_\beta u_{h}^j\nu_{\alpha}\phi_{k}^{i}|_{+}=0,&k=1,2,...,\frac{n(n+1)}{2},\\
u_{h}=\varphi,&\mathrm{on}\ \partial{D},
\end{cases}
\end{align}
where $u_{h}=(u_{h}^{1},...,u_{h}^{n})\in H^{1}(D;\mathbb{R}^{n})$, $\varphi\in C^{2,\alpha}(\overline{D};\mathbb{R}^{n})$, the free constants $C_{i\alpha}^{h}$, $i=1,2,\,\alpha=1,2,...,\frac{n(n+1)}{2}$ are determined by the fifth and sixth lines of \eqref{La.002}, $\nu=(\nu_{1},...,\nu_{n})$ represents the unit outer normal vector of $\partial D_{1}^{h}$ and $\partial D_{2}$, $\{\phi_{\alpha}\}_{\alpha=1}^{\frac{n(n+1)}{2}}$ is a basis of the linear space of rigid displacement $\Phi:=\{\phi\in C^1(\mathbb{R}^{n}; \mathbb{R}^{n})\,|\,\nabla\phi+(\nabla\phi)^T=0\}$, whose elements are given by $\{e_{i},\,x_{k}e_{j}-x_{j}e_{k}\,|\,1\leq i\leq n,\,1\leq j<k\leq n\}$, and
\begin{align*}
A_{ij}^{\alpha\beta}=\lambda\delta_{i\alpha}\delta_{j\beta}+\mu(\delta_{i\beta}\delta_{\alpha j}+\delta_{ij}\delta_{\alpha\beta}),\quad i,j,k,l=1,2,...,n.
\end{align*}
Here $\mu,n\lambda+2\mu\in(0,\infty)$, $\delta_{ij}$ is the kronecker symbol, that is, $\delta_{ij}=0$ if $i\neq j$, $\delta_{ij}=1$ if $i=j$. Similarly, we have

\begin{theorem}\label{thm005}
For $n\geq2$ and $|h|\ll1/2$, let $u_{h}\in H^{1}(D;\mathbb{R}^{n})\cap C^{2,\gamma}(\overline{\Omega}_{h};\mathbb{R}^{n})$ be the solution of the elasticity problem \eqref{La.002} with $D,D_{1}^{h},D_{2},\Omega_{h}$ defined by \eqref{DOM001} and $\gamma$ given by \eqref{alpha}. Then there exists a small constant $h_{0}>0$ such that $u_{h}$ are smooth and stable in $h$ for $h\in(-h_{0},h_{0})$. Furthermore, $\|u_{h}-u_{0}\|_{L^{\infty}(\overline{D})}=O(|D_{1}^{h}\Delta D_{1}^{0}|)$.
\end{theorem}

In the following, we will establish the regularity and stability of solutions to the conductivity problems with degenerate coefficients and the elasticity problem with respect to small translational distance in Sections \ref{SEC02} and \ref{SEC03}, respectively.

\section{Regularity and stability for the conductivity problem with degenerate coefficients}\label{SEC02}
For $m\geq2$, define
\begin{align*}
\widetilde{m}=
\begin{cases}
[m]+1,&\text{if $[m]$ is odd},\\
m,&\text{if $m$ is even},\\
[m]+2,&\text{if $m\neq[m]$ and $[m]$ is even},
\end{cases}
\end{align*}
where $[m]$ denotes the integer part of $m$. For $h,s,t\in\mathbb{R}$ and $0<|h|\ll1/2$, define
\begin{align*}
g(s)=&
\begin{cases}
e^{\frac{1}{(s-(1+|h|)^{\widetilde{m}})(s-(3/2)^{\widetilde{m}})}},&\mathrm{if}\;(1+|h|)^{\widetilde{m}}<s<(3/2)^{\widetilde{m}},\\
0,&\mathrm{if}\;s\leq(1+|h|)^{\widetilde{m}}\;\mathrm{or}\;s\geq(3/2)^{\widetilde{m}},
\end{cases}\\
G(t)=&\frac{\int^{\infty}_{t}g(s)ds}{\int^{\infty}_{-\infty}g(s)ds},\quad\eta(x)=G\left(\sum^{n-1}_{i=1}|x_{i}|^{\widetilde{m}}+|x_{n}-3|^{\widetilde{m}}\right).
\end{align*}
Then $\eta\in C^{\infty}(\mathbb{R}^{n})$ is a cutoff function satisfying that $\eta=1$ in $\mathcal{B}_{1+|h|}^{(\widetilde{m})}(3e_{n})$, $\eta=0$ in $\mathbb{R}^{n}\setminus\mathcal{B}_{3/2}^{(\widetilde{m})}(3e_{n})$ and $|\nabla\eta|\leq C(m,n).$ Let $y=\psi(x)=(x',x_{n}-h\eta(x))$. Then $\psi$ is a diffeomorphism from $\mathbb{R}^{n}$ to $\mathbb{R}^{n}$, as $|h|$ is sufficiently small. In fact, since $\det(\partial_{x}y)=1-h\partial_{x_{n}}\eta$, it suffices to pick $\delta_{0}:=\delta_{0}(m,n)=2^{-1}(\sup\limits_{x\in\mathbb{R}^{n}}|\partial_{x_{n}}\eta|)^{-1}$ and let $h\in(-\delta_{0},\delta_{0})$. Remark that this diffeomorphism carry out translation and split joint in the interior of $D$.

%\subsection{The perfect and insulated conductivity problem}
Rewrite the original problems \eqref{con006}, \eqref{PRO002} and \eqref{PRO09} as follows:
\begin{align}\label{con002}
\begin{cases}
\mathrm{div}(a_{h}^{k}(x)\nabla u_{h}^{k})=0,&\mathrm{in}\; D,\\
u_{h}^{k}=\varphi, &\mathrm{on}\;\partial D,
\end{cases}
\end{align}
where the coefficients $a_{k}(x)$, $k=1,2,3$ are, respectively, given by
\begin{align}\label{QKW001}
a_{h}^{1}(x)=&
\begin{cases}
\infty,&x\in D_{1}^{h}\cup D_{2},\\
1,&x\in \Omega_{h},
\end{cases}\quad a_{h}^{2}(x)=
\begin{cases}
0,&x\in D_{1}^{h}\cup D_{2},\\
1,&x\in \Omega_{h},
\end{cases}
\end{align}
and
\begin{align}\label{QKW00100}
a_{h}^{3}(x)=&
\begin{cases}
0,&x\in D_{1}^{h},\\
\infty,&x\in D_{2},\\
1,&x\in\Omega_{h}.
\end{cases}
\end{align}
Denote $\tilde{u}_{h}^{k}(y)=u_{h}^{k}(x)$, $k=1,2,3$. By the above change of variables, problem \eqref{con002} turns into
\begin{align}\label{ZWQ001Z66}
\begin{cases}
\mathrm{div}(\tilde{A}_{h}^{k}(y)\nabla \tilde{u}_{h}^{k})=0,&\mathrm{in}\; D,\\
\tilde{u}_{h}^{k}=\varphi, &\mathrm{on}\;\partial D,
\end{cases}
\end{align}
where $\tilde{A}_{h}^{1}(y)=\infty I$ in $\overline{D_{1}^{0}\cup D_{2}}$, $\tilde{A}_{h}^{2}(y)=0I$ in $\overline{D_{1}^{0}\cup D_{2}}$, $\tilde{A}_{h}^{3}(y)=0I$ in $\overline{D_{1}^{0}}$, $\tilde{A}_{h}^{3}=\infty I$ in $\overline{D}_{2}$, $I$ is the identity matrix, and for $y\in \Omega_{0}$, $k=1,2,3,$
\begin{gather}
\begin{align*}
\tilde{A}_{h}^{k}(y)=&(\tilde{a}_{ij})=\frac{(\partial_{x}y)(\partial_{x}y)^{t}}{\det(\partial_{x}y)}
=\frac{1}{b_{n}}
\begin{pmatrix}1&0&\cdots&0&b_{1} \\ 0&1&\cdots&0&b_{2}\\ \vdots&\vdots&\ddots&\vdots&\vdots\\0&0&\cdots&1&b_{n-1}\\ b_{1}&b_{2}&\cdots&b_{n-1}&\sum^{n-1}_{i=1}b_{i}^{2}+b_{n}^{2}
\end{pmatrix},
\end{align*}
\end{gather}
and
\begin{align*}
b_{i}=-h\partial_{x_{i}}\eta(\psi^{-1}(y)),\;i=1,...,n-1,\quad b_{n}=1-h\partial_{x_{n}}\eta(\psi^{-1}(y)).
\end{align*}

Introduce the following function spaces:
%\begin{align*}
%\begin{cases}
%X^{1}=\{\tilde{v}\in H^{1}(D)\cap C^{2,\gamma}(\overline{\Omega}_{0})\,|\,\text{$\nabla \tilde{v}=0$ in $\overline{D_{1}^{0}\cup D_{2}}$, $\tilde{v}=\varphi$ on $\partial D$}\},\\
%Y^{1}=\{f\in H^{-1}(D)\cap C^{0,\gamma}(\overline{\Omega}_{0})\,|\,\text{$f=0$ in $\overline{D_{1}^{0}\cup D_{2}}$}\},
%\end{cases}
%\end{align*}
\begin{align*}
X^{1}=\Big\{&\tilde{v}\in H^{1}(D)\cap C^{2,\gamma}(\overline{\Omega}_{0})\,\big|\,\text{$\nabla \tilde{v}=0$ in $\overline{D_{1}^{0}\cup D_{2}}$, $\tilde{v}|_{+}=\tilde{v}|_{-}$ on $\partial D_{1}^{0}\cup\partial D_{2}$,}\notag\\
&\int_{\partial D_{1}^{0}}\frac{\partial\tilde{v}}{\partial\nu}\Big|_{+}=\int_{\partial D_{2}}\frac{\partial \tilde{v}}{\partial\nu}\Big|_{+}=0,\;\text{$\tilde{v}=\varphi$ on $\partial D$}\Big\},\\
Y^{1}=\{&f\in H^{-1}(D)\cap C^{0,\gamma}(\overline{\Omega}_{0})\,|\,\text{$f=0$ in $\overline{D_{1}^{0}\cup D_{2}}$}\},
\end{align*}
and
\begin{align*}
X^{2}=\Big\{&\tilde{v}\in H^{1}(D)\cap C^{2,\gamma}(\overline{\Omega}_{0})\cap C^{2,\gamma}(\overline{D_{1}^{0}\cup D_{2}})\,|\,
\tilde{v}|_{+}=\tilde{v}|_{-}\;\text{on}\; \partial D_{1}^{0}\cup\partial D_{2},\notag\\
&\frac{\partial u}{\partial\nu}\Big|_{+}=0\;\text{on}\; \partial D_{1}^{0}\cup\partial D_{2},\;\text{$\tilde{v}=\varphi$ on $\partial D$}\Big\},\\
Y^{2}=\{&f\in H^{-1}(D)\cap C^{0,\gamma}(\overline{\Omega}_{0})\cap C^{0,\gamma}(\overline{D_{1}^{0}\cup D_{2}})\},
\end{align*}
and
\begin{align*}
X^{3}=\Big\{&\tilde{v}\in H^{1}(D)\cap C^{2,\gamma}(\overline{\Omega}_{0})\cap C^{2,\gamma}(\overline{D_{1}^{0}})\,|\,\text{$\nabla \tilde{v}=0$ in $\overline{D}_{2}$, $\tilde{v}|_{+}=\tilde{v}|_{-}$ on $\partial D_{1}^{0}\cup\partial D_{2}$,}\notag\\
&\text{$\frac{\partial u}{\partial\nu}\Big|_{+}=0$ on $\partial D_{1}^{0}$, $\int_{\partial D_{2}}\frac{\partial u}{\partial\nu}\Big|_{+}=0$,  $\tilde{v}=\varphi$ on $\partial D$}\Big\},\\
Y^{3}=\{&f\in H^{-1}(D)\cap C^{0,\gamma}(\overline{\Omega}_{0})\cap C^{0,\gamma}(\overline{D_{1}^{0}})\,|\,\text{$f=0$ in $\overline{D}_{2}$}\},
\end{align*}
with the corresponding norms as $\|\tilde{v}\|_{X^{k}}:=\|\tilde{v}\|_{H^{1}(D)}+\|\nabla\tilde{v}\|_{L^{\infty}(\frac{1}{2}D)}$ and $\|f\|_{Y^{k}}:=\|f\|_{H^{-1}(D)}=\sup\{\langle f,w\rangle\,|\,w\in H^{1}_{0}(D),\,\|w\|_{H^{1}_{0}(D)}\leq1\}$, $k=1,2,3,$ where $\gamma$ is defined by \eqref{alpha}, $\langle\, ,\rangle$ denotes the pairing between $H^{-1}(D)$ and $H_{0}^{1}(D)$. It is easy to verify that $X^{k}$ and $Y^{k}$ are Banach spaces. Denote $K:=(-\delta_{0},\delta_{0})$ and $F^{k}(h,\tilde{v}):=\mathrm{div}(\tilde{A}_{h}^{k}(y)\nabla \tilde{v})$, $k=1,2,3$.

\begin{prop}\label{QAT001}
For each $k=1,2,3,$ the map $F^{k}$ is in $C^{\infty}(K\times X^{k},Y^{k})$ in the sense that $F^{k}$ possesses continuous Fr\'{e}chet derivatives of any order.

\end{prop}

\begin{proof}
Observe first that for any $s\geq0$,
\begin{align}\label{MZ001}
\partial_{h}^{s}F^{1}(h,\tilde{v})=
\begin{cases}
\partial_{i}(\partial_{h}^{s}\tilde{a}_{ij}\partial_{j}\tilde{v}),&\mathrm{in}\;\Omega_{0},\\
0,&\mathrm{in}\;D\setminus\Omega_{0},
\end{cases}
\end{align}
and
\begin{align}\label{MZ006}
F^{2}(h,\tilde{v})=
\begin{cases}
\partial_{i}(\tilde{a}_{ij}\partial_{j}\tilde{v}),&\mathrm{in}\;\Omega_{0},\\
\Delta \tilde{v},&\mathrm{in}\;D\setminus\Omega_{0},
\end{cases}\;F^{3}(h,\tilde{v})=
\begin{cases}
\partial_{i}(\tilde{a}_{ij}\partial_{j}\tilde{v}),&\mathrm{in}\;\Omega_{0},\\
\Delta \tilde{v},&\mathrm{in}\;\overline{D_{1}^{0}},\\
0,&\mathrm{in}\;\overline{D}_{2},
\end{cases}
\end{align}
and, for $s\geq1$, $k=2,3,$
\begin{align}\label{MZ009}
\partial_{h}^{s}F^{k}(h,\tilde{v})=
\begin{cases}
\partial_{i}(\partial_{h}^{s}\tilde{a}_{ij}\partial_{j}\tilde{v}),&\mathrm{in}\;\Omega_{0},\\
0,&\mathrm{in}\;D\setminus\Omega_{0}.
\end{cases}
\end{align}
Note that for each $h\in K$, $\partial_{h}^{s}F^{k}(h,\tilde{v})$ is linear with respect to $\tilde{v}$, where $s\geq0,\,k=1,2,3$. According to the definition of $\tilde{a}_{ij}$, it follows from a direct computation that $|\partial_{h}^{s}\tilde{a}_{ij}|\leq C(s,m,n)$ in $\overline{K\times\Omega_{0}}$. This, together with the fact that $\tilde{v}\in X^{k}$, shows that $\partial_{h}^{s}F^{k}(h,\tilde{v})\in L^{2}(D)$. Hence, for any $w\in H_{0}^{1}(D)$, $\|w\|_{H^{1}_{0}(D)}\leq1$,

$(i)$ if $k=1$ and $s\geq0$, we have from \eqref{MZ001}, integration by parts, H\"{o}lder's inequality and Trace Theorem that
\begin{align*}
|\langle \partial_{h}^{s}F^{1},w\rangle|=&\left|\int_{D}\partial_{h}^{s}F^{1}(h,\tilde{v})w\,dy\right|=\left|\int_{\Omega_{0}}\partial_{h}^{s}F^{1}(h,\tilde{v})w\,dy\right|\notag\\
=&\left|\int_{\partial D_{1}^{0}\cup\partial D_{2}}\partial_{h}^{s}\tilde{a}_{ij}\partial_{j}\tilde{v}\nu_{i}w-\int_{\Omega_{0}}\partial_{h}^{s}\tilde{a}_{ij}\partial_{j}\tilde{v}\partial_{i}wdy\right|\notag\\
\leq&C(\|\nabla\tilde{v}\|_{L^{\infty}(\frac{1}{2}D)}\|w\|_{L^{2}(\partial D_{1}^{0}\cup\partial D_{2})}+\|\nabla \tilde{v}\|_{L^{2}(\Omega_{0})}\|\nabla w\|_{L^{2}(D)})\notag\\
\leq&C(\|\nabla\tilde{v}\|_{L^{\infty}(\frac{1}{2}D)}+\|\nabla\tilde{v}\|_{L^{2}(\Omega_{0})})\|w\|_{H^{1}_{0}(D)}\leq C\|\tilde{v}\|_{X^{1}};
\end{align*}

$(ii)$ for $k=2,3$, if $s=0$, it follows from \eqref{MZ006}--\eqref{MZ009}, integration by parts, Trace Theorem and H\"{o}lder's inequality again that
\begin{align*}
|\langle F^{2},w\rangle|=&\left|\int_{D}F^{2}(h,\tilde{v})w\,dy\right|=\left|\int_{\Omega_{0}}\partial_{i}(\tilde{a}_{ij}\partial_{j}\tilde{v})w\,dy+\int_{D_{1}^{0}\cup D_{2}}\Delta\tilde{v}w\right|\notag\\
=&\left|\int_{\partial D_{1}^{0}\cup\partial D_{2}}\left(\frac{\partial\tilde{v}}{\partial\nu}\Big|_{+}+\frac{\partial\tilde{v}}{\partial\nu}\Big|_{-}\right)w-\int_{\Omega_{0}}\tilde{a}_{ij}\partial_{j}\tilde{v}\partial_{i}w-\int_{D_{1}^{0}\cup D_{2}}\nabla\tilde{v}\nabla w\right|\notag\\
\leq& C(\|\nabla\tilde{v}\|_{L^{\infty}(\frac{1}{2}D)}+\|\nabla\tilde{v}\|_{L^{2}(D)})\|w\|_{H^{1}_{0}(D)}\leq C\|\tilde{v}\|_{X^{2}},
\end{align*}
and
\begin{align*}
|\langle F^{3},w\rangle|=&\left|\int_{D}F^{3}(h,\tilde{v})w\,dy\right|=\left|\int_{\Omega_{0}}\partial_{i}(\tilde{a}_{ij}\partial_{j}\tilde{v})w\,dy+\int_{D_{1}^{0}}\Delta\tilde{v}w\right|\notag\\
=&\left|\int_{\partial D_{1}^{0}\cup\partial D_{2}}\frac{\partial\tilde{v}}{\partial\nu}\Big|_{+}\cdot w+\int_{\partial D_{1}^{0}}\frac{\partial\tilde{v}}{\partial\nu}\Big|_{-}\cdot w-\int_{\Omega_{0}}\tilde{a}_{ij}\partial_{j}\tilde{v}\partial_{i}w-\int_{D_{1}^{0}}\nabla\tilde{v}\nabla w\right|\notag\\
\leq& C(\|\nabla\tilde{v}\|_{L^{\infty}(\frac{1}{2}D)}+\|\nabla\tilde{v}\|_{L^{2}(\Omega_{0}\cup D_{1}^{0})})\|w\|_{H^{1}_{0}(D)}\leq C\|\tilde{v}\|_{X^{3}},
\end{align*}
and, if $s\geq1$, using the fact that $\partial_{h}^{s}\tilde{a}_{ij}=0$ on $\partial D_{1}^{0}\cup\partial D_{2}$,
\begin{align*}
|\langle\partial_{h}^{s}F^{k},w\rangle|=&\left|\int_{D}\partial_{h}^{s}F^{k}(h,\tilde{v})w\,dy\right|=\left|\int_{\Omega_{0}}\partial_{i}(\partial_{h}^{s}\tilde{a}_{ij}\partial_{j}\tilde{v})w\,dy\right|=\left|\int_{\Omega_{0}}\partial_{h}^{s}\tilde{a}_{ij}\partial_{j}\tilde{v}\partial_{i}w\right|\notag\\
\leq& C\|\nabla\tilde{v}\|_{L^{2}(\Omega_{0})}\|\nabla w\|_{L^{2}(\Omega_{0})}\leq C\|\tilde{v}\|_{X^{k}},
\end{align*}
Therefore, we deduce that for $k=1,2,3,$ $\|\partial_{h}^{s}F^{k}(h,\tilde{v})\|_{Y^{k}}\leq C\|\tilde{v}\|_{X^{k}}$ for all $\tilde{v}\in X^{k}$. Then for any $h\in K$, $\partial_{h}^{s}F^{k}(h,\cdot):X^{k}\rightarrow Y^{k}$ is a bounded linear operator with uniformly bounded norm on $K$. Hence we have from standard theories in functional analysis that for $k=1,2,3,$ $F^{k}$ is a $C^{\infty}$ map from $K\times X^{k}$ to $Y^{k}$.

\end{proof}

Note that for $k=1,2,3,$ $F^{k}(h,\tilde{u}+\tilde{v})-F^{k}(h,\tilde{u})=F^{k}(h,\tilde{v}):=\mathcal{L}^{h,k}_{\tilde{u}}\tilde{v}$, where $h\in K$ and $\tilde{u},\tilde{v}\in X^{k}$. Then we obtain that the linear bounded operator $\mathcal{L}^{h,k}_{\tilde{u}}:X^{k}\rightarrow Y^{k}$ is the Fr\'{e}chet derivative of $F^{k}$ with respect to $\tilde{u}$ at $(h,\tilde{u}).$ When $h=0$, $y=\psi(x)=x$. Then we have $\mathcal{L}^{0,k}_{u_{0}^{k}}\tilde{v}=\mathrm{div}(A_{0}^{k}\nabla\tilde{v})$, where $A_{0}^{k}=a_{0}^{k}(x)I$ with $a_{0}^{k}(x)$ defined by \eqref{QKW001}--\eqref{QKW00100} with $h=0$. In particular, $\mathcal{L}^{0,k}_{u_{0}^{k}}u_{0}^{k}=\mathrm{div}(A_{0}^{k}\nabla u_{0}^{k})=0.$

\begin{prop}\label{QAT002}
For every $k=1,2,3,$ the operator $\mathcal{L}^{0,k}_{u_{0}^{k}}:X^{k}\rightarrow Y^{k}$ is an isomorphism.
\end{prop}

\begin{proof}
For $k=1,2,3$, if $\tilde{v}_{i}\in X^{k}$, $i=1,2$ satisfy $\mathcal{L}^{0,k}_{u_{0}^{k}}\tilde{v}_{1}=\mathcal{L}^{0,k}_{u_{0}^{k}}\tilde{v}_{2}$, then $\mathrm{div}(A_{0}^{k}\nabla(\tilde{v}_{1}-\tilde{v}_{2}))=0$. Since $\tilde{v}_{1}-\tilde{v}_{2}=0$ on $\partial D$, it then follows from the uniqueness of weak solution that $\tilde{v}_{1}=\tilde{v}_{2}$ in $D$. Then for $k=1,2,3,$ the map $\mathcal{L}^{0,k}_{u_{0}^{k}}:X^{k}\rightarrow Y^{k}$ is injective. In the following we divide into three subcases to verify that for $k=1,2,3,$ $\mathcal{L}^{0,k}_{u_{0}^{k}}:X^{k}\rightarrow Y^{k}$ is surjective.

%{\bf Case 1.} Consider the case when $k=1.$ Choose $g\in C^{2}(\overline{D})$ such that $g=C_{1}^{0}$ in $D_{1}^{0}$, $g=C_{2}^{0}$ in $D_{2}$, and $g=\varphi$ on $\partial D$, where the constants $C_{i}^{0}$, $i=1,2$ are defined by the fourth line of problem \eqref{con006}. Then $\mathrm{div}(A_{0}^{1}\nabla g)\in H^{-1}(D)$. For any $f\in Y^{1}$, we have $f-\mathrm{div}(A_{0}^{1}\nabla g)\in H^{-1}(D)\cap C^{0,\gamma}(\overline{\Omega}_{0})$ and thus $f-\mathrm{div}(A_{0}^{1}\nabla g)\in H^{-1}(\Omega_{0})\cap C^{0,\gamma}(\overline{\Omega}_{0})$. Since $\Delta:H^{1}_{0}(\Omega_{0})\rightarrow H^{-1}(\Omega_{0})$ is an isomorphism, there exists a unique weak solution $w\in H^{1}_{0}(\Omega_{0})$ such that $\mathrm{div}(A_{0}^{1}\nabla w)=f-\mathrm{div}(A_{0}^{1}\nabla g)$ in $\Omega_{0}$. In view of $f-\mathrm{div}(A_{0}^{1}\nabla g)\in C^{0,\gamma}(\overline{\Omega}_{0}),$ we deduce from the elliptic regularity theories that $w\in C^{2,\gamma}(\overline{\Omega}_{0}).$ Set
%\begin{align*}
%\tilde{w}=
%\begin{cases}
%w,&\mathrm{in}\;\Omega_{0},\\
%0,&\mathrm{in}\;D\setminus\Omega_{0},
%\end{cases}\quad \tilde{u}=\tilde{w}+g,\,\;\mathrm{in}\;D.
%\end{align*}
%Therefore, $\tilde{u}\in X^{1}$ satisfies $\mathcal{L}^{0,1}_{u_{0}^{1}}\tilde{u}=f$.

{\bf Case 1.}
For any $f\in Y^{1}$, let $\tilde{v}_{i}\in C^{2,\gamma}(\overline{\Omega}_{0})$, $i=1,2,3$ be, respectively, the solutions of the following equations
\begin{align*}
\begin{cases}
\Delta \tilde{v}_{1}=0,&\mathrm{in}\;\Omega_{0},\\
\tilde{v}_{1}=1,&\mathrm{on}\;\overline{D_{1}^{0}},\\
\tilde{v}_{2}=0,&\mathrm{on}\;\overline{D}_{2}\cup\partial D,
\end{cases}\quad\begin{cases}
\Delta \tilde{v}_{2}=0,&\mathrm{in}\;\Omega_{0},\\
\tilde{v}_{2}=1,&\mathrm{on}\;\overline{D}_{2},\\
\tilde{v}_{2}=0,&\mathrm{on}\;\overline{D_{1}^{0}}\cup\partial D,
\end{cases}
\end{align*}
and
\begin{align*}
\begin{cases}
\Delta \tilde{v}_{3}=f,&\mathrm{in}\;\Omega_{0},\\
\tilde{v}_{3}=0,&\mathrm{on}\;\overline{D_{1}^{0}\cup D_{2}},\\
\tilde{v}_{3}=\varphi,&\mathrm{on}\;\partial D.
\end{cases}
\end{align*}
For $j=1,2,$ denote
\begin{align*}
\tilde{a}_{1j}=\int_{\partial D_{1}^{0}}\frac{\partial \tilde{v}_{j}}{\partial\nu}\Big|_{+},\quad \tilde{a}_{2j}=\int_{\partial D_{2}}\frac{\partial \tilde{v}_{j}}{\partial\nu}\Big|_{+},
\end{align*}
and
\begin{align*}
\tilde{b}_{1}:=\tilde{b}_{1}[f,\varphi]=\int_{\partial D_{1}^{0}}\frac{\partial \tilde{v}_{3}}{\partial\nu}\Big|_{+},\quad \tilde{b}_{2}:=\tilde{b}_{2}[f,\varphi]=\int_{\partial D_{2}}\frac{\partial \tilde{v}_{3}}{\partial\nu}\Big|_{+}.
\end{align*}
Let
\begin{align*}
\widetilde{C}_{1}:=\widetilde{C}_{1}[f,\varphi]=\frac{\tilde{a}_{12}\tilde{b}_{2}-\tilde{b}_{1}\tilde{a}_{22}}{\tilde{a}_{11}\tilde{a}_{22}-\tilde{a}_{12}\tilde{a}_{21}},\quad \widetilde{C}_{2}:=\widetilde{C}_{2}[f,\varphi]=\frac{\tilde{a}_{21}\tilde{b}_{1}-\tilde{b}_{2}\tilde{a}_{11}}{\tilde{a}_{11}\tilde{a}_{22}-\tilde{a}_{12}\tilde{a}_{21}}.
\end{align*}
In view of $f\in C^{0,\gamma}(\overline{\Omega}_{0})$, it then follows from the classical elliptic theories that there exists a unique solution $\tilde{u}\in H^{1}(D)\cap C^{2,\gamma}(\overline{\Omega}_{0})$ such that
\begin{align*}
\begin{cases}
\Delta \tilde{u}=f,&\mathrm{in}\;\Omega_{0},\\
\tilde{u}=\widetilde{C}_{1},&\mathrm{on}\;\overline{D_{1}^{0}},\\
\tilde{u}=\widetilde{C}_{2},&\mathrm{on}\;\overline{D_{2}},\\
\tilde{u}=\varphi,&\mathrm{on}\;\partial D.
\end{cases}
\end{align*}
Note that $\tilde{u}=\widetilde{C}_{1}\tilde{v}_{1}+\widetilde{C}_{2}\tilde{v}_{2}+\tilde{v}_{3}$ in $\Omega_{0}$. This, together with a direct calculation, shows that $\int_{\partial D_{1}^{0}}\frac{\partial\tilde{u}}{\partial\nu}\big|_{+}=\int_{\partial D_{2}}\frac{\partial\tilde{u}}{\partial\nu}\big|_{+}=0$. Therefore, $\tilde{u}\in X^{1}$ satisfies $\mathcal{L}^{0,1}_{u_{0}^{1}}\tilde{u}=f$.

{\bf Case 2.} For any $f\in Y^{2}$, we have from the classical elliptic theories that there exists a unique solution $\tilde{u}_{1}\in C^{2,\gamma}(\overline{\Omega}_{0})$ for equation $\mathrm{div}(A_{0}^{2}\nabla \tilde{u}_{1})=f$ satisfying that $\tilde{u}_{1}=\varphi$ on $\partial D$ and $\frac{\partial\tilde{u}_{1}}{\partial\nu}|_{+}=0$ on $\partial D_{1}^{0}\cup \partial D_{2}$. Since $\tilde{u}_{1}\in C^{2,\gamma}(\partial D_{1}^{0}\cup \partial D_{2})$ and $f\in C^{0,\gamma}(\overline{D_{1}^{0}\cup D_{2}})$, it then follows from the elliptic theories again that there exists unique solutions $\tilde{u}_{2}\in C^{2,\gamma}(\overline{D_{1}^{0}})$ and $\tilde{u}_{3}\in C^{2,\gamma}(\overline{D}_{2})$ such that
\begin{align*}
\begin{cases}
\Delta\tilde{u}_{2}=f,&\mathrm{in}\; D_{1}^{0},\\
\tilde{u}_{2}=\tilde{u}_{1},&\mathrm{on}\;\partial D_{1}^{0},
\end{cases}\quad
\begin{cases}
\Delta\tilde{u}_{3}=f,&\mathrm{in}\; D_{2},\\
\tilde{u}_{3}=\tilde{u}_{1},&\mathrm{on}\;\partial D_{2}.
\end{cases}
\end{align*}
Take
\begin{align*}
\tilde{u}=
\begin{cases}
\tilde{u}_{1},&\mathrm{in}\;\overline{\Omega}_{0},\\
\tilde{u}_{2},&\mathrm{in}\;D_{1}^{0},\\
\tilde{u}_{3},&\mathrm{in}\; D_{2}.
\end{cases}
\end{align*}
Then $\tilde{u}\in X^{2}$ solves $\mathcal{L}^{0,2}_{u_{0}^{2}}\tilde{u}=f$.

{\bf Case 3.} For any $f\in Y^{3}$, let $\tilde{v}_{i}\in C^{2,\gamma}(\overline{\Omega}_{0})$, $i=1,2$, respectively, satisfy
\begin{align*}
\begin{cases}
\Delta \tilde{v}_{1}=0,&\mathrm{in}\;\Omega_{0},\\
\frac{\partial \tilde{v}_{1}}{\partial\nu}\big|_{+}=0,&\mathrm{on}\;\partial D_{1}^{0},\\
\tilde{v}_{1}=1,&\mathrm{on}\;\overline{D}_{2},\\
\tilde{v}_{1}=0,&\mathrm{on}\;\partial D,
\end{cases}\quad\begin{cases}
\Delta \tilde{v}_{2}=f,&\mathrm{in}\;\Omega_{0},\\
\frac{\partial \tilde{v}_{2}}{\partial\nu}\big|_{+}=0,&\mathrm{on}\;\partial D_{1}^{0},\\
\tilde{v}_{2}=0,&\mathrm{on}\;\overline{D}_{2},\\
\tilde{v}_{2}=\varphi,&\mathrm{on}\;\partial D.
\end{cases}
\end{align*}
Define
\begin{align*}
\widetilde{C}_{0}:=\frac{\widetilde{Q}[f,\varphi]}{\tilde{a}_{21}},\quad\tilde{a}_{21}=\int_{\partial D_{2}}\frac{\partial \tilde{v}_{1}}{\partial\nu}\Big|_{+},\;\, \widetilde{Q}[f,\varphi]=-\int_{\partial D_{2}}\frac{\partial \tilde{v}_{2}}{\partial\nu}\Big|_{+}.
\end{align*}
Since $f\in C^{0,\gamma}(\overline{\Omega}_{0})$, we obtain from the classical elliptic theories that there exists a unique solution $\tilde{u}_{1}\in C^{2,\gamma}(\overline{\Omega}_{0})$ satisfying
\begin{align*}
\begin{cases}
\Delta \tilde{u}_{1}=f,&\mathrm{in}\;\Omega_{0},\\
\frac{\partial\tilde{u}_{1}}{\partial\nu}\big|_{+}=0,&\mathrm{on}\;\partial D_{1}^{0},\\
\tilde{u}_{1}=\widetilde{C}_{0},&\mathrm{on}\;\overline{D}_{2},\\
\tilde{u}_{1}=\varphi,&\mathrm{on}\;\partial D.
\end{cases}
\end{align*}
Observe that $\tilde{u}_{1}\in C^{2,\gamma}(\partial D_{1}^{0})$ and $f\in C^{0,\gamma}(\overline{D_{1}^{0}})$. Then using the elliptic theories again, we deduce that there exists a unique solution $\tilde{u}_{2}\in C^{2,\gamma}(\overline{D_{1}^{0}})$ such that
\begin{align*}
\begin{cases}
\Delta\tilde{u}_{2}=f,&\mathrm{in}\; D_{1}^{0},\\
\tilde{u}_{2}=\tilde{u}_{1},&\mathrm{on}\;\partial D_{1}^{0}.
\end{cases}
\end{align*}
Let
\begin{align*}
\tilde{u}=
\begin{cases}
\tilde{u}_{1},&\mathrm{in}\;\overline{D\setminus D_{1}^{0}},\\
\tilde{u}_{2},&\mathrm{in}\;D_{1}^{0}.
\end{cases}
\end{align*}
Observe that $\tilde{u}=\widetilde{C}_{0}\tilde{v}_{1}+\tilde{v}_{2}$ in $\Omega_{0}$. Then by a direct computation, we derive $\int_{\partial D_{2}}\frac{\partial\tilde{u}}{\partial\nu}\big|_{+}=0$. Hence, $\tilde{u}\in X^{3}$ solves $\mathcal{L}^{0,3}_{u_{0}^{3}}\tilde{u}=f$. The proof is complete.

\end{proof}

We are now ready to give the proof of Theorem \ref{thm001}.

\begin{proof}[Proof of Theorem \ref{thm001}]
Combining Propositions \ref{QAT001} and \ref{QAT002}, we deduce from the implicit function theorem (see Theorem 2.7.2 in \cite{N2001}) that for $k=1,2,3,$ there exists a small positive constant $h_{k}=h_{k}(\Omega_{0},\|\varphi\|_{C^{2}(\partial D)})$ such that $\tilde{u}^{k}_{h}\in C^{\infty}((-h_{k},h_{k}))$, where $\tilde{u}_{h}^{k}$ is the solution of \eqref{ZWQ001Z66}. By taking $u^{k}_{h}(x)=\tilde{u}^{k}_{h}(y)$ with $y=\psi(x)$, we obtain that $u^{k}_{h}\in C^{\infty}((-h_{k},h_{k}))$, where $u_{h}^{k}$ satisfies equation \eqref{con002}. It remains to establish the stability of $u^{k}_{h}$ with respect to $h$. For the convenience of notations, let $v^{k}(h,x):=u^{k}_{h}(x)$ in the following.

{\bf Case 1.} Consider the case of $k=1$. For any fixed $h\in(-h_{1},h_{1})$, the solution $v^{1}(h,x)$ of problem \eqref{con002} can be split as follows:
\begin{align}\label{DEC001}
v^{1}(h,x)=C_{1}^{h}v_{1}(h,x)+C_{2}^{h}v_{2}(h,x)+v_{0}(h,x),\quad\mathrm{in}\;\Omega_{h},
\end{align}
where $v_{i}$, $i=0,1,2$, respectively, solve
\begin{align*}
\begin{cases}
\Delta_{x} v_{0}=0,&\mathrm{in}\;\Omega_{h},\\
v_{0}=0,&\mathrm{on}\;\overline{D_{1}^{h}\cup D_{2}},\\
v_{0}=\varphi,&\mathrm{on}\;\partial D,
\end{cases}
\end{align*}
and
\begin{align*}
\begin{cases}
\Delta_{x} v_{1}=0,&\mathrm{in}\;\Omega_{h},\\
v_{1}=1,&\mathrm{on}\;\overline{D_{1}^{h}},\\
v_{1}=0,&\mathrm{on}\;\overline{D}_{2}\cup\partial D,
\end{cases}\quad\begin{cases}
\Delta_{x} v_{2}=0,&\mathrm{in}\;\Omega_{h},\\
v_{2}=1,&\mathrm{on}\;\overline{D}_{2},\\
v_{2}=0,&\mathrm{on}\;\overline{D_{1}^{h}}\cup\partial D.
\end{cases}
\end{align*}
Substituting \eqref{DEC001} into the fourth line of \eqref{con006}, we have
\begin{align*}
\begin{cases}
a_{11}^{h}C_{1}^{h}+a_{12}^{h}C_{2}^{h}+b_{1}^{h}=0,\\
a_{21}^{h}C_{1}^{h}+a_{22}^{h}C_{2}^{h}+b_{2}^{h}=0,
\end{cases}
\end{align*}
where $a_{ij}^{h}$ and $b_{i}^{h}$, $i,j=1,2,$ are defined by
\begin{align*}
a_{1j}^{h}=\int_{\partial D_{1}^{h}}\frac{\partial v_{j}}{\partial\nu},\quad a_{2j}^{h}=\int_{\partial D_{2}}\frac{\partial v_{j}}{\partial\nu},\quad b_{1}^{h}=\int_{\partial D_{1}^{h}}\frac{\partial v_{0}}{\partial\nu},\quad b_{2}^{h}=\int_{\partial D_{2}}\frac{\partial v_{0}}{\partial\nu}.
\end{align*}
By Cramer's rule, we derive
\begin{align*}
C_{1}^{h}=\frac{a_{12}^{h}b_{2}^{h}-b_{1}^{h}a_{22}^{h}}{a_{11}^{h}a_{22}^{h}-a_{12}^{h}a_{21}^{h}},\quad C_{2}^{h}=\frac{a_{21}^{h}b_{1}^{h}-b_{2}^{h}a_{11}^{h}}{a_{11}^{h}a_{22}^{h}-a_{12}^{h}a_{21}^{h}}.
\end{align*}
Integrating by parts, we obtain
\begin{align*}
a_{ij}^{h}=\int_{\Omega_{h}}\nabla v_{i}\nabla v_{j},\quad b_{i}^{h}=\int_{\Omega_{h}}\nabla v_{0}\nabla v_{i},\quad i,j=1,2.
\end{align*}

From the mean value theorem and the boundary estimates for elliptic equations, we obtain that for some $\theta_{i}\in(0,1)$, $i=0,1,2,$

$(1)$ for $x\in(\partial D_{1}^{h}\setminus D_{1}^{0})\cap\{x_{n}\geq3+h\}$ if $h\geq0$, or for $x\in(\partial D_{1}^{h}\setminus D_{1}^{0})\cap\{x_{n}<3+h\}$ if $h<0,$
\begin{align}\label{DEF90681}
|v_{i}(h,x)-v_{i}(0,x)|=&|v_{i}(0,x',x_{n}-h)-v_{i}(0,x',x_{n})|\notag\\
=&|\partial_{x_{n}} v_{i}(0,x',x_{n}-\theta_{i}h)||h|\leq C|h|;
\end{align}

$(2)$ for $x\in(\partial D_{1}^{h}\setminus D_{1}^{0})\cap\{1-(|h|/2)^{m}\leq\sum^{n-1}_{j=1}|x_{j}|^{m}\leq1\}\cap\{x_{n}\leq3+h\}$ if $h\geq0$, or for $x\in(\partial D_{1}^{h}\setminus D_{1}^{0})\cap\{1-(|h|/2)^{m}\leq\sum^{n-1}_{j=1}|x_{j}|^{m}\leq1\}\cap\{x_{n}>3+h\}$ if $h<0$,
\begin{align}\label{QMWD9086}
|v_{i}(h,x)-v_{i}(0,x)|=&\bigg|v_{i}\bigg(0,x',3+\mathrm{sgn}(h)\bigg(1-\sum^{n-1}_{j=1}|x_{i}|^{m}\bigg)^{1/m}\bigg)-v_{i}(0,x)\bigg|\notag\\
=&|\partial_{x_{n}} v_{i}(0,x',x_{n}-\theta_{i}h)|\bigg(|h|-2\bigg(1-\sum^{n-1}_{j=1}|x_{i}|^{m}\bigg)^{1/m}\bigg)\notag\\
\leq& C\bigg(|h|-2\bigg(1-\sum^{n-1}_{j=1}|x_{i}|^{m}\bigg)^{1/m}\bigg),
\end{align}
where $\mathrm{sgn}$ is the sign function such that $\mathrm{sgn}(h)=1$ if $h>0$, and $\mathrm{sgn}(h)=-1$ if $h<0$. In contrast to \eqref{DEF90681}, we can capture a smaller pointwise difference between $v_{i}(h,x)$ and $v_{i}(0,x)$ on the boundary near the corner in \eqref{QMWD9086}. In exactly the same way, we deduce that for $x\in\partial D_{1}^{0}\setminus D_{1}^{h}$,
\begin{align*}
|v_{i}(h,x)-v_{i}(0,x)|=|v_{i}(h,x',x_{n})-v_{i}(h,x',x_{n}+h)|\leq C|h|.
\end{align*}
These two relations, in combination with the fact that $v_{i}(h,x)-v_{i}(0,x)=0$ on $\partial D_{2}\cup\partial D$, gives that
\begin{align*}
|v_{i}(h,x)-v_{i}(0,x)|\leq C|h|,\quad\mathrm{on}\;\partial(D\setminus\overline{D_{1}^{0}\cup D_{1}^{h}\cup D_{2}}).
\end{align*}
By the maximum principle, we further have
\begin{align*}
|v_{i}(h,x)-v_{i}(0,x)|\leq C|h|,\quad\mathrm{in}\;D\setminus\overline{D_{1}^{0}\cup D_{1}^{h}\cup D_{2}}.
\end{align*}
For $i=0,1,2,$ denote
\begin{align*}
V_{i}(h,x)=\frac{v_{i}(h,x)}{h},\quad h\in(-h_{1},h_{1}).
\end{align*}
Then we have $|V_{i}(h,x)-V_{i}(0,x)|\leq C\,\;\mathrm{in}\;D\setminus\overline{D_{1}^{0}\cup D_{1}^{h}\cup D_{2}}.$ This, together with the standard elliptic estimates, shows that
\begin{align*}
\|V_{i}(h,\cdot)-V_{i}(0,\cdot)\|_{C^{2}(\overline{D\setminus (D_{1}^{0}\cup D_{1}^{h}\cup D_{2})})}\leq C,
\end{align*}
and thus
\begin{align*}
\|v_{i}(h,\cdot)-v_{i}(0,\cdot)\|_{C^{2}(\overline{D\setminus (D_{1}^{0}\cup D_{1}^{h}\cup D_{2})})}\leq C|h|,
\end{align*}
which leads to that for $i,j=1,2,$
\begin{align*}
a_{ij}^{h}=&\int_{\Omega_{h}}\nabla v_{i}(h,x)\nabla v_{j}(h,x)=\int_{D\setminus\overline{D_{1}^{0}\cup D_{1}^{h}\cup D_{2}}}\nabla v_{i}(h,x)\nabla v_{j}(h,x)+O(|h|)\notag\\
=&\int_{D\setminus\overline{D_{1}^{0}\cup D_{1}^{h}\cup D_{2}}}[\nabla v_{i}(0,x)\nabla v_{j}(0,x)+(\nabla v_{i}(h,x)-\nabla v_{i}(0,x))\nabla v_{j}(h,x)]\notag\\
&+\int_{D\setminus\overline{D_{1}^{0}\cup D_{1}^{h}\cup D_{2}}}\nabla v_{i}(0,x)(\nabla v_{j}(h,x)-\nabla v_{j}(0,x))+O(|h|)\notag\\
=&a_{ij}^{0}+O(|h|).
\end{align*}
By the same argument, we also have $b_{i}^{h}=b_{i}^{0}+O(|h|)$, $i=1,2.$ Hence we obtain
\begin{align}\label{DZQAM001}
C_{i}^{h}=C_{i}^{0}+O(|h|),\quad i=1,2.
\end{align}
Observe that for $i=0,1,2,$ $v_{i}(h,x)=v_{i}(0,x)$ on $\overline{D_{1}^{0}\cap D_{1}^{h}}$. Then for $x\in D_{1}^{h}\setminus\overline{D_{1}^{0}}$, we deduce from the standard elliptic estimates that
\begin{align*}
&|v_{i}(h,x',x_{n})-v_{i}(0,x',x_{n})|\notag\\
&=\bigg|v_{i}(0,x',x_{n})-v_{i}\Big(0,x',3+\mathrm{sgn}(h)\Big(1-\sum^{n-1}_{i=1}|x_{i}|^{m}\Big)^{1/m}\Big)\bigg|\leq C|h|,
\end{align*}
where $\mathrm{sgn}$ is the sign function with respect to $h$. In exactly the same way, we obtain that for $x\in D_{1}^{0}\setminus\overline{D_{1}^{h}}$,
\begin{align*}
&|v_{i}(h,x',x_{n})-v_{i}(0,x',x_{n})|\notag\\
&=\bigg|v_{i}(h,x',x_{n})-v_{i}\Big(h,x',3+h-\mathrm{sgn}(h)\Big(1-\sum^{n-1}_{i=1}|x_{i}|^{m}\Big)^{1/m}\Big)\bigg|\leq C|h|.
\end{align*}
Combining these above facts, we obtain that for $h\in(-h_{1},h_{1})$,
\begin{align}\label{ZMQ006}
\|v_{i}(h,\cdot)-v_{i}(0,\cdot)\|_{L^{\infty}(\overline{D})}=O(|h|),\quad i=0,1,2.
\end{align}
Observe that
\begin{align}\label{ZMQ009}
v^{1}(h,x)-v^{1}(0,x)=&\sum^{2}_{i=1}\left(C_{i}^{h}(v_{i}(h,x)-v_{i}(0,x))+(C_{i}^{h}-C_{i}^{0})v_{i}(0,x)\right)\notag\\
&+v_{0}(h,x)-v_{0}(0,x).
\end{align}
Then inserting \eqref{DZQAM001}--\eqref{ZMQ006} into \eqref{ZMQ009}, we deduce
\begin{align*}
\|v^{1}(h,\cdot)-v^{1}(0,\cdot)\|_{L^{\infty}(\overline{D})}=O(|h|),\quad h\in(-h_{1},h_{1}).
\end{align*}

{\bf Case 2.} Consider the case of $k=2$. To begin with, for $h\in(-h_{2},h_{2})$, in view of $\frac{\partial v^{2}(0,x)}{\partial\nu}|_{+}=0$ on $\partial D_{1}^{0}\cup \partial D_{2}$ and $\frac{\partial v^{2}(h,x)}{\partial\nu}|_{+}=0$ on $\partial D_{1}^{h}\cup\partial D_{2}$, it follows from the standard elliptic estimates that $x\in \partial D_{1}^{h}\setminus D_{1}^{0}$,
\begin{align*}
\left|\frac{\partial (v^{2}(h,x)-v^{2}(0,x))}{\partial\nu}\Big|_{+}\right|=&\left|\frac{\partial v^{2}(0,x',x_{n}-h)}{\partial\nu}\Big|_{+}-\frac{\partial v^{2}(0,x',x_{n})}{\partial\nu}\Big|_{+}\right|\notag\\
\leq&|\nabla(v^{2}(0,x',x_{n}-h)-v^{2}(0,x',x_{n}))|\leq C|h|,
\end{align*}
and, for $x\in\partial D_{1}^{0}\setminus D_{1}^{h}$,
\begin{align*}
\left|\frac{\partial (v^{2}(h,x)-v^{2}(0,x))}{\partial\nu}\Big|_{+}\right|=&\left|\frac{\partial v^{2}(h,x',x_{n})}{\partial\nu}\Big|_{+}-\frac{\partial v^{2}(h,x',x_{n}+h)}{\partial\nu}\Big|_{+}\right|\notag\\
\leq&|\nabla(v^{2}(h,x',x_{n})-v^{2}(h,x',x_{n}+h))|\leq C|h|.
\end{align*}
Then we have
\begin{align}\label{NED001}
\frac{\partial (v^{2}(h,x)-v^{2}(0,x))}{\partial\nu}\Big|_{+}=
\begin{cases}
O(|h|),& \mathrm{on}\;(\partial D_{1}^{h}\setminus D_{1}^{0})\cup(\partial D_{1}^{0}\setminus D_{1}^{h}),\\
0,&\mathrm{on}\;\partial D_{2}.
\end{cases}
\end{align}
For simplicity, denote $w(h,x):=v^{2}(h,x)-v^{2}(0,x)$ and $\widetilde{\Omega}_{h}:=D\setminus\overline{D_{1}^{0}\cup D_{1}^{h}\cup D_{2}}$. Then $w(h,x)$ solves
\begin{align*}
\begin{cases}
\Delta_{x} w(h,x)=0,&\mathrm{in}\;\widetilde{\Omega}_{h},\\
w(h,x)=0,&\mathrm{on}\;\partial D.
\end{cases}
\end{align*}
Multiplying this equation by $|w(h,x)|^{p-2}w(h,x)$ with $p\geq2$, it then follows from \eqref{NED001}, integration by parts, H\"{o}lder's inequality, Trace Theorem and Poincar\'{e}'s inequality that for $h\in(-h_{2},h_{2}),$
\begin{align*}
&(p-1)\int_{\widetilde{\Omega}_{h}}|\nabla w(h,x)|^{2}|w(h,x)|^{p-2}=\int_{\partial\widetilde{\Omega}_{h}}\frac{\partial w(h,x)}{\partial\nu}\Big|_{+}\cdot w(h,x)|w(h,x)|^{p-2}\notag\\
&\leq C|h|\int_{(\partial D_{1}^{h}\setminus D_{1}^{0})\cup(\partial D_{1}^{0}\setminus D_{1}^{h})}|w(h,x)|^{p-1}\leq C|h|\||w(h,\cdot)|^{p-1}\|_{W^{1,1}(\widetilde{\Omega}_{h})}\notag\\
&\leq C|h|\|\nabla |w(h,\cdot)|^{p-1}\|_{L^{1}(\widetilde{\Omega}_{h})}=C|h|(p-1)\int_{\widetilde{\Omega}_{h}}|\nabla w(h,x)||w(h,x)|^{p-2}\notag\\
&\leq C|h|(p-1)\left(\int_{\widetilde{\Omega}_{h}}|\nabla w(h,x)|^{2}|w(h,x)|^{p-2}\right)^{\frac{1}{2}}\left(\int_{\widetilde{\Omega}_{h}}|w(h,x)|^{p-2}\right)^{\frac{1}{2}}\notag\\
&\leq C|h|(p-1)\left(\int_{\widetilde{\Omega}_{h}}|\nabla w(h,x)|^{2}|w(h,x)|^{p-2}\right)^{\frac{1}{2}}\|w(h,\cdot)\|_{L^{p}(\widetilde{\Omega}_{h})}^{\frac{p-2}{2}},
\end{align*}
which implies that
\begin{align}\label{ZWQA05}
\left(\int_{\widetilde{\Omega}_{h}}\big|\nabla|w(h,x)|^{\frac{p}{2}}\big|^{2}\right)^{\frac{1}{2}}\leq Cp|h|\|w(h,\cdot)\|_{L^{p}(\widetilde{\Omega}_{h})}^{\frac{p-2}{2}}.
\end{align}
%In particular, when $p=2$, we have from \eqref{ZWQA05} and Poincar\'{e}'s inequality that
%\begin{align}\label{ZWQA06}
%\|w(h,\cdot)\|_{L^{2}(\widetilde{\Omega}_{h})}\leq C\|\nabla w(h,\cdot)\|_{L^{2}(\widetilde{\Omega}_{h})}\leq Ch.
%\end{align}
Applying Sobolev-Poincar\'{e} inequalities to \eqref{ZWQA05}, we obtain
\begin{align}\label{ZWQA08}
\|w(h,\cdot)\|_{L^{tp}(\widetilde{\Omega}_{h})}\leq& (Cp^{2})^{1/p}|h|^{2/p}\|w(h,\cdot)\|_{L^{p}(\widetilde{\Omega}_{h})}^{\frac{p-2}{p}},
%\leq&(Cp^{2})^{1/p}h^{2/p}\left(\frac{p-2}{p}\|w(h,\cdot)\|_{L^{p}(\widetilde{\Omega}_{h})}+\frac{2}{p}\right),
\end{align}
where $t:=t(n)$ is given by
\begin{align*}
\begin{cases}
t>1,&n=2,\\
t=\frac{n}{n-2},&n>2.
\end{cases}
\end{align*}
Let
\begin{align*}
p_{k}=2t^{k},\quad k\geq0,\,n\geq2.
\end{align*}
By iteration with \eqref{ZWQA08}, we deduce that for $k>2$,
\begin{align*}
\|w(h,\cdot)\|_{L^{p_{k}}(\widetilde{\Omega}_{h})}\leq&\prod^{k-1}_{i=0}(Cp_{i}^{2})^{1/p_{i}}|h|^{\frac{2}{p_{k-1}}+\sum\limits^{k-2}_{j=0}\frac{2}{p_{j}}\prod\limits^{k-1}_{l=j+1}(1-\frac{2}{p_{l}})}\notag\\
\leq& C|h| t^{\sum\limits^{k-1}_{i=0}\frac{i}{t^{i}}}\leq C|h|,
\end{align*}
where $C$ is independent of $k$. By sending $k\rightarrow\infty$, we obtain
\begin{align}\label{AMYW001}
\|v^{2}(h,\cdot)-v^{2}(0,\cdot)\|_{L^{\infty}(\overline{D\setminus(D_{1}^{0}\cup D_{1}^{h}\cup D_{2})})}\leq C|h|.
\end{align}
It is worthwhile to emphasize that \eqref{AMYW001} is proved by using Moser's iteration argument. For $x=(x',x_{n})\in D_{1}^{h}\setminus D_{1}^{0}$, denote
\begin{align*}
I_{1}:=&v^{2}\Big(h,x',3+h+\mathrm{sgn}(h)\Big(1-\sum^{n-1}_{i=1}|x_{i}|^{m}\Big)^{1/m}\Big)\notag\\
&-v^{2}\Big(0,x',3+h+\mathrm{sgn}(h)\Big(1-\sum^{n-1}_{i=1}|x_{i}|^{m}\Big)^{1/m}\Big),\\
I_{2}:=&v^{2}(h,x',x_{n})-v^{2}\Big(h,x',3+h+\mathrm{sgn}(h)\Big(1-\sum^{n-1}_{i=1}|x_{i}|^{m}\Big)^{1/m}\Big),\\
I_{3}:=&v^{2}\Big(0,x',3+h+\mathrm{sgn}(h)\Big(1-\sum^{n-1}_{i=1}|x_{i}|^{m}\Big)^{1/m}\Big)-v^{2}(0,x',x_{n}).
\end{align*}
From \eqref{AMYW001}, we have $|I_{1}|\leq C|h|$. Applying the standard elliptic estimates for $v^{2}(h,x)$ in $\overline{D_{1}^{h}}$ and $v^{2}(0,x)$ in $\overline{D\setminus(D_{1}^{0}\cup D_{2})}$, we obtain that $|I_{2}|\leq C|h|$ and $I_{3}\leq C|h|$. Then we obtain
\begin{align*}
\|v^{2}(h,\cdot)-v^{2}(0,\cdot)\|_{L^{\infty}(\overline{D_{1}^{h}\setminus D_{1}^{0}})}\leq C|h|.
\end{align*}
In exactly the same way, we also have $\|v^{2}(h,x)-v^{2}(0,h)\|_{L^{\infty}(\overline{D_{1}^{0}\setminus D_{1}^{h}})}\leq C|h|.$ Hence we have
\begin{align}\label{ZZW006}
\|v^{2}(h,\cdot)-v^{2}(0,\cdot)\|_{L^{\infty}(\overline{(D_{1}^{h}\setminus D_{1}^{0})\cup(D_{1}^{0}\setminus D_{1}^{h})})}\leq C|h|.
\end{align}
Then applying the maximum principle for $v^{2}(h,x)-v^{2}(0,x)$ in $\overline{D_{1}^{0}\cap D_{1}^{h}}$, we derive
\begin{align*}
\|v^{2}(h,\cdot)-v^{2}(0,\cdot)\|_{L^{\infty}(\overline{D_{1}^{0}\cap D_{1}^{h}})}\leq C|h|,
\end{align*}
which, together with \eqref{AMYW001}--\eqref{ZZW006}, reads that
\begin{align}\label{QLAZM006}
\|v^{2}(h,\cdot)-v^{2}(0,\cdot)\|_{L^{\infty}(\overline{D})}\leq C|h|,\quad h\in(-h_{2},h_{2}).
\end{align}

{\bf Case 3.} Consider the case when $k=3$. Similarly as before, for any given $h\in(-h_{3},h_{3})$, we decompose the solution $v^{3}(h,x)$ of problem \eqref{PRO09} as follows:
\begin{align}\label{DEC006}
v^{3}(h,x)=C^{h}v_{1}(h,x)+v_{2}(h,x),\quad\mathrm{in}\;\Omega_{h},
\end{align}
where $v_{i}$, $i=1,2$, respectively, satisfy
\begin{align*}
\begin{cases}
\Delta_{x} v_{1}(h,x)=0,&\mathrm{in}\;D\setminus(\partial D_{1}^{h}\cup\overline{D}_{2}),\\
\frac{\partial v_{1}(h,x)}{\partial\nu}\big|_{+}=0,&\mathrm{on}\;\partial D_{1}^{h},\\
v_{1}(h,x)=1,&\mathrm{on}\;\overline{D}_{2},\\
v_{1}(h,x)=0,&\mathrm{on}\;\partial D,
\end{cases}\quad\begin{cases}
\Delta_{x} v_{2}(h,x)=0,&\mathrm{in}\;D\setminus(\partial D_{1}^{h}\cup\overline{D}_{2}),\\
\frac{\partial v_{2}(h,x)}{\partial\nu}\big|_{+}=0,&\mathrm{on}\;\partial D_{1}^{h},\\
v_{2}(h,x)=0,&\mathrm{on}\;\overline{D}_{2},\\
v_{2}(h,x)=\varphi,&\mathrm{on}\;\partial D.
\end{cases}
\end{align*}
Inserting \eqref{DEC006} into the fifth line of \eqref{PRO09}, we obtain
\begin{align*}
a_{21}^{h}C^{h}=Q^{h}[\varphi],\quad\text{and thus } C^{h}=\frac{Q^{h}[\varphi]}{a_{21}^{h}},
\end{align*}
where $a_{21}^{h}$ and $Q^{h}[\varphi]$ are given by
\begin{align*}
a_{21}^{h}=\int_{\partial D_{2}}\frac{\partial v_{1}(h,x)}{\partial\nu},\quad Q^{h}[\varphi]=-\int_{\partial D_{2}}\frac{\partial v_{2}(h,x)}{\partial\nu}.
\end{align*}
From integration by parts, we see
\begin{align*}
a_{21}^{h}=\int_{\Omega_{h}}|\nabla v_{1}(h,x)|^{2},\quad Q^{h}[\varphi]=-\int_{\Omega_{h}}\nabla v_{1}(h,x)\nabla v_{2}(h,x).
\end{align*}
By the same arguments as in \eqref{QLAZM006}, it follows from Moser's iteration argument, the standard elliptic estimates and the maximum principle that for $i=1,2,$
\begin{align*}
\|v_{i}(h,\cdot)-v_{i}(0,\cdot)\|_{L^{\infty}(\overline{D})}\leq C|h|,
\end{align*}
which, in combination with the rescale argument and the standard elliptic estimates, reads that
\begin{align*}
\|v_{i}(h,\cdot)-v_{i}(0,\cdot)\|_{C^{2}(\overline{D\setminus (D_{1}^{0}\cup D_{1}^{h}\cup D_{2})})}\leq C|h|.
\end{align*}
This leads to that
\begin{align*}
a_{21}^{h}=a_{21}^{0}+O(|h|),\quad Q^{h}[\varphi]=Q^{0}[\varphi]+O(|h|),
\end{align*}
and we thus have $C^{h}=C^{0}+O(|h|)$. Note that
\begin{align*}
v^{3}(h,x)-v^{3}(0,x)=&C^{h}(v_{1}(h,x)-v_{1}(0,x))+(C^{h}-C^{0})v_{1}(0,x)\notag\\
&+v_{2}(h,x)-v_{2}(0,x).
\end{align*}
Therefore, combining these above facts, we deduce
\begin{align*}
\|v^{3}(h,\cdot)-v^{3}(0,\cdot)\|_{L^{\infty}(\overline{D})}\leq C|h|.
\end{align*}
The proof is complete.

\end{proof}

\section{Regularity and stability for the elasticity problem}\label{SEC03}

As shown in the appendix of \cite{BLL2015}, problem \eqref{La.002} is the limiting case of the following problem
\begin{align}\label{ZWE006}
\begin{cases}
\partial_\alpha(A_{ij}^{\alpha\beta}\partial_\beta u_{h}^j)=0,&\mathrm{in}\; D,\\
u_{h}=\varphi, &\mathrm{on}\;\partial D,
\end{cases}
\end{align}
with
\begin{align*}
A_{ij}^{\alpha\beta}(x)=
\begin{cases}
\lambda_{1}\delta_{i\alpha}\delta_{j\beta}+\mu_{1}(\delta_{i\beta}\delta_{\alpha j}+\delta_{ij}\delta_{\alpha\beta}),&\mathrm{in}\;D_{1}^{h}\cup D_{2},\\
\lambda\delta_{i\alpha}\delta_{j\beta}+\mu(\delta_{i\beta}\delta_{\alpha j}+\delta_{ij}\delta_{\alpha\beta}),&\mathrm{in}\;\Omega_{h},
\end{cases}
\end{align*}
as $\mu_{1},n\lambda_{1}+2\mu_{1}\rightarrow\infty$. So problem \eqref{La.002} can be rewritten as \eqref{ZWE006} with $(\lambda_{1},\mu_{1})$ and $(\lambda,\mu)$ satisfying that $\mu_{1}=n\lambda_{1}+2\mu_{1}=\infty$, $\mu,n\lambda+2\mu\in(0,\infty)$.

Denote $\tilde{u}_{h}(y)=u_{h}(x)$, where $y=\psi(x)$. Then the original problem \eqref{La.002} becomes
\begin{align}\label{ZWQ001}
\begin{cases}
\partial_\alpha(\tilde{A}_{ij}^{\alpha\beta}\partial_\beta \tilde{u}_{h}^j)=0,&\mathrm{in}\; D,\\
\tilde{u}_{h}=\varphi, &\mathrm{on}\;\partial D,
\end{cases}
\end{align}
where, for $i,j,\alpha,\beta=1,2,...,n,$ $\tilde{A}_{ij}^{\alpha\beta}|_{D_{1}^{0}\cup D_{2}}=A_{ij}^{\alpha\beta}|_{D_{1}^{h}\cup D_{2}}$, and for $i,j=1,2,...,n,$
\begin{align*}
(\tilde{A}_{ij}^{\alpha\beta}(y))=\frac{(\partial_{x}y)(A_{ij}^{\alpha\beta})(\partial_{x}y)^{t}}{\det(\partial_{x}y)}=(\lambda\delta_{i\alpha}\delta_{j\beta}+\mu(\delta_{i\beta}\delta_{\alpha j}+\delta_{ij}\delta_{\alpha\beta}))+O(h),\quad\mathrm{in}\;\Omega_{0}.
\end{align*}
Similarly as above, define
\begin{align*}
X=\Big\{&\tilde{v}\in H^{1}(D;\mathbb{R}^{n})\cap C^{2,\gamma}(\overline{\Omega}_{0};\mathbb{R}^{n})\,|\,\text{$\nabla\tilde{v}+(\nabla\tilde{v})^{T}=0$ in $\overline{D_{1}^{0}\cup D_{2}}$, $\tilde{v}=\varphi$ on $\partial D$},\notag\\
&u|_{+}=u|_{-}\text{ on }\partial D_{1}^{0}\cap\partial D_{2},\;\int_{\partial{D}_{1}^{0}}A_{ij}^{\alpha\beta}\partial_\beta u_{h}^j\nu_{\alpha}\phi_{k}^{i}|_{+}=0,\notag\\
&\int_{\partial{D}_{2}}A_{ij}^{\alpha\beta}\partial_\beta u_{h}^j\nu_{\alpha}\phi_{k}^{i}|_{+}=0,\;k=1,2,...,\frac{n(n+1)}{2}\Big\},\notag\\
Y=\{&f\in H^{-1}(D;\mathbb{R}^{n})\cap C^{0,\gamma}(\overline{\Omega}_{0};\mathbb{R}^{n})\,|\,\text{$f=0$ in $\overline{D_{1}^{0}\cup D_{2}}$}\},
\end{align*}
with their norms as $\|\tilde{v}\|_{X}:=\|\tilde{v}\|_{H^{1}(D;\mathbb{R}^{n})}+\|\nabla\tilde{v}\|_{L^{\infty}(\frac{1}{2}D;\mathbb{R}^{n})}$ and
$$\|f\|_{Y}:=\|f\|_{H^{-1}(D;\mathbb{R}^{n})}=\sup\{\langle f,w\rangle\,|\,w\in H^{1}_{0}(D;\mathbb{R}^{n}),\,\|w\|_{H^{1}_{0}(D;\mathbb{R}^{n})}\leq1\},$$
where $\gamma$ is given by \eqref{alpha}, $\langle\, ,\rangle$ denotes the pairing between $H^{-1}(D;\mathbb{R}^{n})$ and $H_{0}^{1}(D;\mathbb{R}^{n})$. It is not difficult to demonstrate that $X$ and $Y$ are Banach spaces. Define $K:=(-\delta_{0},\delta_{0})$ and $F(h,\tilde{v}):=(F_{1}(h,\tilde{v}),...,F_{n}(h,\tilde{v}))=(\partial_\alpha(\tilde{A}_{1j}^{\alpha\beta}\partial_\beta \tilde{v}^j),...,\partial_\alpha(\tilde{A}_{nj}^{\alpha\beta}\partial_\beta \tilde{v}^j)).$

\begin{lemma}\label{QAT001V2}
The map $F$ is in $C^{\infty}(K\times X,Y)$ in the sense that $F$ possesses continuous Fr\'{e}chet derivatives of any order.

\end{lemma}

\begin{proof}
To begin with, for any $k\geq0$, $h\in K$ and $\tilde{v}\in X$, we have
\begin{align*}
\partial_{h}^{k}F(h,\tilde{v})=
\begin{cases}
(\partial_\alpha(\partial_{h}^{k}\tilde{A}_{1j}^{\alpha\beta}\partial_\beta \tilde{v}^j),...,\partial_\alpha(\partial_{h}^{k}\tilde{A}_{nj}^{\alpha\beta}\partial_\beta \tilde{v}^j)),& \mathrm{in}\; \Omega_{0},\\
(0,...,0),&\mathrm{in}\;D\setminus\Omega_{0}.
\end{cases}
\end{align*}
Observe that for each $h\in K$, $\partial_{h}^{k}F$ is linear in $\tilde{v}$. By a direct calculation, we obtain that for $i,j,\alpha,\beta=1,2,...,n,$ $|\partial_{h}^{k}\tilde{A}_{ij}^{\alpha\beta}|\leq C(k,m,n)$ in $\overline{K\times\Omega_{0}}$. This, in combination with the fact that $\tilde{v}\in X$, implies that $\partial_{h}^{k}F(h,\tilde{v})\in L^{2}(D;\mathbb{R}^{n})$. Then we deduce from integration by parts, H\"{o}lder's inequality and Trace Theorem that for all $w\in H_{0}^{1}(D;\mathbb{R}^{n})$, $\|w\|_{H^{1}_{0}(D;\mathbb{R}^{n})}\leq1$,
\begin{align*}
|\langle \partial_{h}^{k}F,w\rangle|=&\left|\int_{D}\partial_{h}^{k}F(h,\tilde{v})w\,dy\right|=\left|\int_{\Omega_{0}}\partial_{h}^{k}F(h,\tilde{v})w\,dy\right|\notag\\
=&\left|\int_{\partial D_{1}^{0}\cup\partial D_{2}}\partial_{h}^{k}\tilde{A}_{ij}^{\alpha\beta}\partial_\beta \tilde{v}^j\nu_{\alpha}w^{i}-\int_{\Omega_{0}}\partial_{h}^{k}\tilde{A}_{ij}^{\alpha\beta}\partial_\beta \tilde{v}^j\partial_{\alpha}w^{i}dy\right|\notag\\
\leq&C(\|\nabla\tilde{v}\|_{L^{\infty}(\frac{1}{2}D;\mathbb{R}^{n})}\|w\|_{L^{2}(\partial D_{1}^{0}\cup\partial D_{2})}+\|\nabla \tilde{v}\|_{L^{2}(\Omega_{0};\mathbb{R}^{n})}\|\nabla w\|_{L^{2}(D;\mathbb{R}^{n})})\notag\\
\leq& C(\|\nabla\tilde{v}\|_{L^{\infty}(\frac{1}{2}D;\mathbb{R}^{n})}+\|\nabla\tilde{v}\|_{L^{2}(\Omega_{0};\mathbb{R}^{n})})\|w\|_{H^{1}_{0}(D;\mathbb{R}^{n})}\leq C\|\tilde{v}\|_{X}.
\end{align*}
This yields that $\|\partial_{h}^{k}F(h,\tilde{v})\|_{Y}\leq C(k,m,n)\|\tilde{v}\|_{X}$ for all $\tilde{v}\in X$. Hence, for any $h\in K$, $\partial_{h}^{k}F(h,\cdot):X\rightarrow Y$ is a bounded linear operator with uniformly bounded norm on $K$. Then it follows from standard theories in functional analysis that $F$ is a $C^{\infty}$ map from $K\times X$ to $Y$.

\end{proof}

Observe that $F(h,\tilde{u}+\tilde{v})-F(h,\tilde{u})=F(h,\tilde{v}):=\mathcal{L}^{h}_{\tilde{u}}\tilde{v}$, where $h\in K$ and $\tilde{u},\tilde{v}\in X$. Therefore, we know that the linear bounded operator $\mathcal{L}^{h}_{\tilde{u}}:X\rightarrow Y$ is the Fr\'{e}chet derivative of $F$ with respect to $\tilde{u}$ at $(h,\tilde{u}).$ Moreover, we have $\mathcal{L}^{0}_{u_{0}}\tilde{v}=(\partial_\alpha(A_{1j}^{\alpha\beta}\partial_\beta \tilde{v}^j),...,\partial_\alpha(A_{nj}^{\alpha\beta}\partial_\beta \tilde{v}^j))$ and $\mathcal{L}^{0}_{u_{0}}u_{0}=0.$ For later convenience, we first rewrite the original problem \eqref{La.002} as follows:
\begin{align}\label{AAE001}
\begin{cases}
\nabla\cdot(\mathbb{C}^0e(u_{h}))=0,&\mathrm{in}\;\Omega_{h},\\
u_{h}|_{+}=u_{h}|_{-},&\mathrm{on}\ \partial D_{1}^{h}\cup\partial D_{2},\\
u_{h}=\sum^{\frac{n(n+1)}{2}}_{\alpha=1}C_{1\alpha}^{h}\phi_{\alpha},&\mathrm{on}\;\overline{D_{1}^{h}},\\
u_{h}=\sum^{\frac{n(n+1)}{2}}_{\alpha=1}C_{2\alpha}^{h}\phi_{\alpha},&\mathrm{on}\;\overline{D}_{2},\\
\int_{\partial{D}_{1}^{h}}\frac{\partial u_{h}}{\partial \nu_0}\big|_{+}\cdot\phi_{\alpha}=0,&\alpha=1,2,...,\frac{n(n+1)}{2},\\
\int_{\partial{D}_{2}}\frac{\partial u_{h}}{\partial \nu_0}\big|_{+}\cdot\phi_{\alpha}=0,&\alpha=1,2,...,\frac{n(n+1)}{2},\\
u_{h}=\varphi,&\hbox{on}\ \partial{D},
\end{cases}
\end{align}
where $e(u_{h})=\frac{1}{2}(\nabla u_{h}+(\nabla u_{h})^{T})$ is the strain tensor, $\mathbb{C}^{0}=(C^{0}_{ijkl})$ represents the elasticity tensor satisfying $C_{ijkl}^{0}=\lambda\delta_{ij}\delta_{kl} +\mu(\delta_{ik}\delta_{jl}+\delta_{il}\delta_{jk})$, $(\mathbb{C}^{0}e(u_{h}))_{ij}=\sum_{k,l=1}^{n}C_{ijkl}^{0}(e(u_{h}))_{kl}$, $\{\phi_{\alpha}\}_{\alpha=1}^{\frac{n(n+1)}{2}}$ denotes a basis of the linear space of rigid displacement $\Phi:=\{\phi\in C^1(\mathbb{R}^{n}; \mathbb{R}^{n})\,|\,\nabla\phi+(\nabla\phi)^T=0\}$, and
\begin{align}\label{NOTA01}
\frac{\partial u_{h}}{\partial \nu_0}\Big|_{+}&:=(\mathbb{C}^0e(u_{h}))\nu=\lambda(\nabla\cdot v)\nu+\mu(\nabla u_{h}+(\nabla u_{h})^T)\nu.
\end{align}

\begin{lemma}\label{QAT002V2}
The operator $\mathcal{L}^{0}_{u_{0}}:X\rightarrow Y$ is an isomorphism.
\end{lemma}

\begin{proof}
Similar to Proposition \ref{QAT002}, we deduce from the uniqueness of weak solution that $\mathcal{L}^{0}_{u_{0}}:X\rightarrow Y$ is injective. It remains to show that $\mathcal{L}^{0}_{u_{0}}:X\rightarrow Y$ is also surjective. For any $f\in Y$, let $\tilde{v}_{0}\in C^{2,\gamma}(\overline{\Omega}_{0};\mathbb{R}^{n})$ and $\tilde{v}_{i\alpha}\in C^{2,\gamma}(\overline{\Omega}_{0};\mathbb{R}^{n})$, $i=1,2,\,\alpha=1,2,...,\frac{n(n+1)}{2}$ be, respectively, the solutions of
\begin{align*}
\begin{cases}
\nabla\cdot(\mathbb{C}^0e(\tilde{v}_{0}))=f,&\mathrm{in}\;\Omega_{0},\\
\tilde{v}_{0}=0,&\mathrm{on}\;\overline{D_{1}^{0}\cup D_{2}},\\
\tilde{v}_{0}=\varphi,&\mathrm{on}\;\partial D,
\end{cases}
\end{align*}
and
\begin{align*}
\begin{cases}
\nabla\cdot(\mathbb{C}^0e(\tilde{v}_{1\alpha}))=0,&\mathrm{in}\;\Omega_{0},\\
\tilde{v}_{1\alpha}=\phi_{\alpha},&\mathrm{on}\;\overline{D_{1}^{0}},\\
\tilde{v}_{1\alpha}=0,&\mathrm{on}\;\overline{D}_{2}\cup\partial D,
\end{cases}\;\,\begin{cases}
\nabla\cdot(\mathbb{C}^0e(\tilde{v}_{2\alpha}))=0,&\mathrm{in}\;\Omega_{0},\\
\tilde{v}_{2\alpha}=\phi_{\alpha},&\mathrm{on}\;\overline{D}_{2},\\
\tilde{v}_{2\alpha}=0,&\mathrm{on}\;\overline{D_{1}^{0}}\cup\partial D,
\end{cases}
\end{align*}
where the elasticity tensor $\mathbb{C}^{0}$ is given in \eqref{AAE001}. For $i=1,2$ and $\alpha,\beta=1,2,...,\frac{n(n+1)}{2}$, define
\begin{align*}
\tilde{a}_{i1\alpha\beta}:=&\int_{\partial{D}_{1}^{0}}\frac{\partial \tilde{v}_{i\alpha}}{\partial \nu_0}\Big|_{+}\cdot\phi_{\beta},\quad \tilde{a}_{i2\alpha\beta}:=\int_{\partial{D}_{2}}\frac{\partial \tilde{v}_{i\alpha}}{\partial \nu_0}\Big|_{+}\cdot\phi_{\beta},\notag\\ \tilde{b}_{1\beta}^{h}:=\tilde{b}_{1\beta}[f,\varphi]=&-\int_{\partial D_{1}^{0}}\frac{\partial \tilde{v}_{0}}{\partial \nu_0}\Big|_{+}\cdot\phi_{\beta},\quad \tilde{b}_{2\beta}:=\tilde{b}_{2\beta}[f,\varphi]=-\int_{\partial D_{2}}\frac{\partial \tilde{v}_{0}}{\partial \nu_0}\Big|_{+}\cdot\phi_{\beta},
\end{align*}
where the notation $\frac{\partial}{\partial\nu_{0}}|_{+}$ is defined by \eqref{NOTA01}. For $i,j=1,2,$ denote $\tilde{\mathbb{A}}_{ij}=(\tilde{a}_{ij\alpha\beta})_{\frac{n(n+1)}{2}\times\frac{n(n+1)}{2}}$ and $Y_{i}=\big(b_{i1}^{h},b_{i2}^{h},...,b_{i\frac{n(n+1)}{2}}^{h}\big)^{T}$. For $i,j=1,2$, and $\alpha=1,2,...,\frac{n(n+1)}{2}$, we let $(Y^{1},Y^{2})^{T}$ substitute for the elements of $\alpha$-th column of the matrix $\tilde{\mathbb{A}}_{ij}$ and obtain new matrix $\tilde{\mathbb{A}}_{ij\alpha}$ as follows:
\begin{gather*}
\tilde{\mathbb{A}}_{ij\alpha}=
\begin{pmatrix}
\tilde{a}_{ij11}&\cdots&\tilde{b}_{j1}&\cdots&\tilde{a}_{ij1\frac{n(n+1)}{2}} \\\\ \vdots&\ddots&\vdots&\ddots&\vdots\\\\ \tilde{a}_{ij\frac{n(n+1)}{2}1}&\cdots&\tilde{b}_{j\frac{n(n+1)}{2}}&\cdots&\tilde{a}_{ij\frac{n(n+1)}{2}\frac{n(n+1)}{2}}
\end{pmatrix}.
\end{gather*}
For $\alpha=1,2,...,\frac{n(n+1)}{2}$, let
\begin{align*}
\widetilde{\mathbb{F}}_{1\alpha}=\begin{pmatrix} \tilde{\mathbb{A}}_{11\alpha}&\tilde{\mathbb{A}}_{21} \\  \tilde{\mathbb{A}}_{12\alpha}&\tilde{\mathbb{A}}_{22}
\end{pmatrix},\;\,\widetilde{\mathbb{F}}_{2\alpha}=\begin{pmatrix} \tilde{\mathbb{A}}_{11}&\tilde{\mathbb{A}}_{21\alpha} \\  \tilde{\mathbb{A}}_{12}&\tilde{\mathbb{A}}_{22\alpha}
\end{pmatrix},\;\, \widetilde{\mathbb{F}}=\begin{pmatrix} \tilde{\mathbb{A}}_{11}&\tilde{\mathbb{A}}_{21} \\  \tilde{\mathbb{A}}_{12}&\tilde{\mathbb{A}}_{22}
\end{pmatrix}.
\end{align*}
Denote
\begin{align*}
\widetilde{C}_{i\alpha}:=\widetilde{C}_{i\alpha}[f,\varphi]=\frac{\det\widetilde{\mathbb{F}}_{i\alpha}}{\det\widetilde{\mathbb{F}}},\quad i=1,2,\,\alpha=1,2,...,\frac{n(n+1)}{2}.
\end{align*}
In light of $f\in C^{0,\gamma}(\overline{\Omega}_{0};\mathbb{R}^{n})$ and utilizing the classical elliptic theories, we obtain that there exists a unique solution $\tilde{v}\in H^{1}(D)\cap C^{2,\gamma}(\overline{\Omega}_{0})$ satisfying the following Dirichlet problem
\begin{align*}
\begin{cases}
\nabla\cdot(\mathbb{C}^0e(\tilde{v}))=0,&\mathrm{in}\;\Omega_{0},\\
\tilde{v}=\sum^{\frac{n(n+1)}{2}}_{\alpha=1}\widetilde{C}_{1\alpha}\phi_{\alpha},&\mathrm{on}\;\overline{D_{1}^{0}},\\
\tilde{v}=\sum^{\frac{n(n+1)}{2}}_{\alpha=1}\widetilde{C}_{2\alpha}\phi_{\alpha},&\mathrm{on}\;\overline{D}_{2},\\
\tilde{v}=\varphi,&\hbox{on}\ \partial{D}.
\end{cases}
\end{align*}
Observe that
\begin{align*}
\tilde{v}=\sum^{\frac{n(n+1)}{2}}_{\alpha=1}\sum^{2}_{i=1}\widetilde{C}_{i\alpha}\tilde{v}_{i\alpha}+\tilde{v}_{0},\quad\mathrm{in}\;\Omega_{0}.
\end{align*}
It then follows from a straightforward computation that $\int_{\partial{D}_{1}^{0}}\frac{\partial \tilde{v}}{\partial \nu_0}\big|_{+}\cdot\phi_{\alpha}=\int_{\partial{D}_{2}}\frac{\partial \tilde{v}}{\partial \nu_0}\big|_{+}\cdot\phi_{\alpha}=0$, $\alpha=1,2,...,\frac{n(n+1)}{2}$. Then $\tilde{v}\in X$ solves $\mathcal{L}^{0}_{u_{0}}\tilde{v}=f$. The proof is finished.

\end{proof}

Based on the above facts, we now present the proof of Theorem \ref{thm005}.
\begin{proof}[Proof of Theorem \ref{thm005}]
A combination of the implicit function theorem and Lemmas \ref{QAT001V2} and \ref{QAT002V2} shows that there exists a small positive constant $h_{0}=h_{0}(\Omega_{0},\|\varphi\|_{C^{2}(\partial D)})$ such that $\tilde{u}_{h}\in C^{\infty}((-h_{0},h_{0});\mathbb{R}^{n})$, where $\tilde{u}_{h}$ is the solution for problem \eqref{ZWQ001}. By letting $u_{h}(x)=\tilde{u}_{h}(y)$ with $y=\psi(x)$, we derive that $u_{h}\in C^{\infty}((-h_{0},h_{0});\mathbb{R}^{n})$ solves equation \eqref{La.002}. We now proceed to study the stability of $u_{h}$ with respect to $h$. For simplicity, denote $v(h,x):=u_{h}(x)$ in the following.

For $h\in(-h_{0},h_{0})$, we split the solution $v(h,x)$ of problem \eqref{AAE001} as follows:
\begin{align}\label{DEC002}
v(h,x)=\sum^{\frac{n(n+1)}{2}}_{\alpha=1}\sum^{2}_{i=1}C_{i\alpha}^{h}v_{i\alpha}(h,x)+v_{0}(h,x),\quad\mathrm{in}\;\Omega_{h},
\end{align}
where $v_{0}$ and $v_{i\alpha}$, $i=1,2,\,\alpha=1,2,...,\frac{n(n+1)}{2}$, respectively, satisfy
\begin{align*}
\begin{cases}
\nabla\cdot(\mathbb{C}^0e(v_{0}))=0,&\mathrm{in}\;\Omega_{h},\\
v_{0}=0,&\mathrm{on}\;\overline{D_{1}^{h}\cup D_{2}},\\
v_{0}=\varphi,&\mathrm{on}\;\partial D,
\end{cases}
\end{align*}
and
\begin{align*}
\begin{cases}
\nabla\cdot(\mathbb{C}^0e(v_{1\alpha}))=0,&\mathrm{in}\;\Omega_{h},\\
v_{1\alpha}=\phi_{\alpha},&\mathrm{on}\;\overline{D_{1}^{h}},\\
v_{1\alpha}=0,&\mathrm{on}\;\overline{D}_{2}\cup\partial D,
\end{cases}\;\,\begin{cases}
\nabla\cdot(\mathbb{C}^0e(v_{2\alpha}))=0,&\mathrm{in}\;\Omega_{h},\\
v_{2\alpha}=\phi_{\alpha},&\mathrm{on}\;\overline{D}_{2},\\
v_{2\alpha}=0,&\mathrm{on}\;\overline{D_{1}^{h}}\cup\partial D.
\end{cases}
\end{align*}
For $i=1,2$ and $\alpha,\beta=1,2,...,\frac{n(n+1)}{2}$, denote
\begin{align*}
a_{i1\alpha\beta}^{h}:=&\int_{\partial{D}_{1}^{h}}\frac{\partial v_{i\alpha}}{\partial \nu_0}\Big|_{+}\cdot\phi_{\beta},\quad a_{i2\alpha\beta}^{h}:=\int_{\partial{D}_{2}}\frac{\partial v_{i\alpha}}{\partial \nu_0}\Big|_{+}\cdot\phi_{\beta} ,\notag\\ b_{1\beta}^{h}:=&-\int_{\partial D_{1}^{h}}\frac{\partial v_{0}}{\partial \nu_0}\Big|_{+}\cdot\phi_{\beta},\quad b_{2\beta}^{h}:=-\int_{\partial D_{2}}\frac{\partial v_{0}}{\partial \nu_0}\Big|_{+}\cdot\phi_{\beta}.
\end{align*}
Then substituting \eqref{DEC002} into the fifth and sixth lines of \eqref{AAE001}, we obtain that for $\beta=1,2,...,\frac{n(n+1)}{2},$
\begin{align}\label{AHNTW009}
\begin{cases}
\sum\limits_{\alpha=1}^{\frac{n(n+1)}{2}}\sum\limits^{2}_{i=1}C_{i\alpha}^{h} a_{i1\alpha\beta}^{h}=b_{1\beta}^{h},\\
\sum\limits_{\alpha=1}^{\frac{n(n+1)}{2}}\sum\limits^{2}_{i=1}C_{i\alpha}^{h} a_{i2\alpha\beta}^{h}=b_{2\beta}^{h}.
\end{cases}
\end{align}
For brevity, denote $\mathbb{A}^{h}_{ij}=(a_{ij\alpha\beta}^{h})_{\frac{n(n+1)}{2}\times\frac{n(n+1)}{2}}$, $i,j=1,2,$ and
\begin{align*}
X_{i}=\big(C_{i1}^{h},C_{i2}^{h},...,C_{i\frac{n(n+1)}{2}}^{h}\big)^{T},\;\,Y_{i}=\big(b_{i1}^{h},b_{i2}^{h},...,b_{i\frac{n(n+1)}{2}}^{h}\big)^{T},\quad i=1,2.
\end{align*}
Therefore, \eqref{AHNTW009} becomes
\begin{gather*}
\begin{pmatrix} \mathbb{A}_{11}^{h}&\mathbb{A}_{21}^{h} \\  \mathbb{A}_{12}^{h}&\mathbb{A}_{22}^{h}
\end{pmatrix}
\begin{pmatrix}
X_{1}\\
X_{2}
\end{pmatrix}=
\begin{pmatrix}
Y_{1}\\
Y_{2}
\end{pmatrix}.
\end{gather*}
For $i,j=1,2,\,\alpha=1,2,...,\frac{n(n+1)}{2}$, by replacing the elements of $\alpha$-th column in the matrix $\mathbb{A}_{ij}^{h}$ by column vector $(Y^{1},Y^{2})^{T}$, we obtain new matrix $\mathbb{A}_{ij\alpha}^{h}$ as follows:
\begin{gather*}
\mathbb{A}_{ij\alpha}^{h}=
\begin{pmatrix}
a_{ij11}^{h}&\cdots&b_{j1}^{h}&\cdots&a_{ij1\frac{n(n+1)}{2}}^{h} \\\\ \vdots&\ddots&\vdots&\ddots&\vdots\\\\a_{ij\frac{n(n+1)}{2}1}^{h}&\cdots&b_{j\frac{n(n+1)}{2}}^{h}&\cdots&a_{ij\frac{n(n+1)}{2}\frac{n(n+1)}{2}}^{h}
\end{pmatrix}.
\end{gather*}
For $\alpha=1,2,...,\frac{n(n+1)}{2}$, write
\begin{align*}
\mathbb{F}^{h}_{1\alpha}=\begin{pmatrix} \mathbb{A}^{h}_{11\alpha}&\mathbb{A}^{h}_{21} \\  \mathbb{A}^{h}_{12\alpha}&\mathbb{A}^{h}_{22}
\end{pmatrix},\;\,\mathbb{F}^{h}_{2\alpha}=\begin{pmatrix} \mathbb{A}^{h}_{11}&\mathbb{A}^{h}_{21\alpha} \\  \mathbb{A}^{h}_{12}&\mathbb{A}^{h}_{22\alpha}
\end{pmatrix},\;\, \mathbb{F}^{h}=\begin{pmatrix} \mathbb{A}^{h}_{11}&\mathbb{A}^{h}_{21} \\  \mathbb{A}^{h}_{12}&\mathbb{A}^{h}_{22}
\end{pmatrix}.
\end{align*}
It then follows from Cramer's rule that
\begin{align*}
C_{i\alpha}^{h}=\frac{\det\mathbb{F}^{h}_{i\alpha}}{\det\mathbb{F}^{h}},\quad i=1,2,\,\alpha=1,2,...,\frac{n(n+1)}{2}.
\end{align*}

The following proof is similar to that in the scalar case above. First, it follows from the mean value theorem and the boundary estimates for elliptic systems that for $x\in\partial D_{1}^{h}\setminus D_{1}^{0}$,
\begin{align*}
|v_{i\alpha}(h,x)-v_{i\alpha}(0,x)|\leq|v_{i\alpha}(0,x',x_{n}-h)-v_{i\alpha}(0,x',x_{n})|+|h|\leq C|h|,
\end{align*}
and for $x\in\partial D_{1}^{0}\setminus D_{1}^{h}$,
\begin{align*}
|v_{i\alpha}(h,x)-v_{i\alpha}(0,x)|\leq|v_{i\alpha}(h,x',x_{n})-v_{i\alpha}(h,x',x_{n}+h)|+|h|\leq C|h|.
\end{align*}
In view of the fact that $v_{i\alpha}(h,x)-v_{i\alpha}(0,x)=0$ on $\partial D_{2}\cup\partial D$, we then have
\begin{align*}
|v_{i\alpha}(h,x)-v_{i\alpha}(0,x)|\leq C|h|,\quad\mathrm{on}\;\partial(D\setminus\overline{D_{1}^{0}\cup D_{1}^{h}\cup D_{2}}).
\end{align*}
From the maximum modulus principle, we derive
\begin{align*}
|v_{i\alpha}(h,x)-v_{i\alpha}(0,x)|\leq C|h|,\quad\mathrm{in}\;D\setminus\overline{D_{1}^{0}\cup D_{1}^{h}\cup D_{2}}.
\end{align*}
Then we deduce from the rescale argument and the standard interior and boundary estimates for elliptic systems that
\begin{align}\label{ZKA001}
\|v_{i\alpha}(h,\cdot)-v_{i\alpha}(0,\cdot)\|_{C^{2}(\overline{D\setminus (D_{1}^{0}\cup D_{1}^{h}\cup D_{2})})}\leq C|h|.
\end{align}
Since $v_{i\alpha}(h,x)=v_{i\alpha}(0,x)$ in $\overline{D_{1}^{0}\cap D_{1}^{h}}$, it then follows from the standard elliptic estimates that for $x\in D_{1}^{h}\setminus\overline{D_{1}^{0}}$,
\begin{align*}
&|v_{i\alpha}(h,x',x_{n})-v_{i\alpha}(0,x',x_{n})|\notag\\
&\leq\bigg|v_{i\alpha}(0,x',x_{n})-v_{i\alpha}\Big(0,x',3+\mathrm{sgn}(h)\Big(1-\sum^{n-1}_{i=1}|x_{i}|^{m}\Big)^{1/m}\Big)\bigg|+|h|\leq C|h|,
\end{align*}
and for $x\in D_{1}^{0}\setminus\overline{D_{1}^{h}}$,
\begin{align*}
&|v_{i\alpha}(h,x',x_{n})-v_{i\alpha}(0,x',x_{n})|\notag\\
&\leq\bigg|v_{i\alpha}(h,x',x_{n})-v_{i\alpha}\Big(h,x',3+h-\mathrm{sgn}(h)\Big(1-\sum^{n-1}_{i=1}|x_{i}|^{m}\Big)^{1/m}\Big)\bigg|+|h|\leq C|h|,
\end{align*}
where $\mathrm{sgn}(h)$ is the sign function in $h$. A consequence of these above facts shows that for $h\in(-h_{0},h_{0})$, $i=1,2,\,\alpha=1,2,...,\frac{n(n+1)}{2}$, $\|v_{i\alpha}(h,\cdot)-v_{i\alpha}(0,\cdot)\|_{L^{\infty}(\overline{D})}=O(|h|)$. By the same argument, we also have $\|v_{0}(h,\cdot)-v_{0}(0,\cdot)\|_{L^{\infty}(\overline{D})}=O(|h|)$.

Observe from integration by parts that for $i,j=1,2$ and $\alpha,\beta=1,2,...,\frac{n(n+1)}{2}$,
\begin{align*}
a_{ij\alpha\beta}^{h}=\int_{\Omega_{h}}(\mathbb{C}^{0}e(v_{i\alpha}),e(v_{j\beta})),\quad b_{i\alpha}^{h}=-\int_{\Omega_{h}}(\mathbb{C}^{0}e(v_{0}),e(v_{i\alpha})).
\end{align*}
From \eqref{ZKA001}, we obtain that for $i,j=1,2,\,\alpha,\beta=1,2,...,\frac{n(n+1)}{2}$,
\begin{align*}
a_{ij\alpha\beta}^{h}=&\int_{D\setminus\overline{D_{1}^{0}\cup D_{1}^{h}\cup D_{2}}}(\mathbb{C}^{0}e(v_{i\alpha}(h,x)),e(v_{j\beta}(h,x)))+O(|h|)\notag\\
=&\int_{D\setminus\overline{D_{1}^{0}\cup D_{1}^{h}\cup D_{2}}}(\mathbb{C}^{0}e(v_{i\alpha}(h,x)-v_{i\alpha}(0,x)),e(v_{j\beta}(h,x)))\notag\\
&+\int_{D\setminus\overline{D_{1}^{0}\cup D_{1}^{h}\cup D_{2}}}(\mathbb{C}^{0}e(v_{i\alpha}(0,x)),e(v_{j\beta}(h,x)-v_{j\beta}(0,x)))\notag\\
&+\int_{D\setminus\overline{D_{1}^{0}\cup D_{1}^{h}\cup D_{2}}}(\mathbb{C}^{0}e(v_{i\alpha}(0,x)),e(v_{j\beta}(0,x)))+O(|h|)\notag\\
=&a_{ij\alpha\beta}^{0}+O(|h|).
\end{align*}
In exactly the same way, we obtain $b_{i\alpha}^{h}=b_{i\alpha}^{0}+O(|h|)$, $i=1,2,\,\alpha=1,2,...,\frac{n(n+1)}{2}$. Combining these above facts, we deduce that $C_{i\alpha}^{h}=C_{i\alpha}^{0}+O(|h|)$, $i=1,2,\,\alpha=1,2,...,\frac{n(n+1)}{2}$. Since
\begin{align*}
v(h,x)-v(0,x)=&\sum^{2}_{i=1}\sum^{\frac{n(n+1)}{2}}_{\alpha=1}\left(C_{i\alpha}^{h}(v_{i\alpha}(h,x)-v_{i\alpha}(0,x))+(C_{i\alpha}^{h}-C_{i\alpha}^{0})v_{i\alpha}(0,x)\right)\notag\\
&+v_{0}(h,x)-v_{0}(0,x),
\end{align*}
it then follows from the above facts that
\begin{align*}
\|v(h,\cdot)-v(0,\cdot)\|_{L^{\infty}(\overline{D})}=O(|h|),\quad h\in(-h_{0},h_{0}).
\end{align*}

\end{proof}

\noindent{\bf{\large Acknowledgements.}}
The author is deeply grateful to Professor JinGang Xiong for asking this question and giving thorough guidance and numerous supports. The author also would like to thank Zhuolun Yang for his useful discussions and suggestions which greatly improve the proofs. The author was partially supported by CPSF (2021M700358).

\bibliographystyle{plain}

\def\cprime{$'$}

\end{document}